\documentclass[a4paper,12pt]{amsart}
\usepackage[T1]{fontenc} 							
\usepackage[ansinew]{inputenc} 							
\usepackage{ellipsis} 							
\usepackage{lmodern} 							
\usepackage{amssymb, amsmath, amsthm}			
\usepackage{bm, bbm}							
\usepackage[fixamsmath]{mathtools} 						

\usepackage[shortlabels]{enumitem} 							
\usepackage[spacing,kerning]{microtype}
\usepackage{geometry}
\usepackage{hyperref}
\usepackage{csquotes}							
\usepackage{todonotes}
\usepackage{subcaption}

\allowdisplaybreaks
\microtypecontext{spacing=nonfrench}

\theoremstyle{plain}
\newtheorem{thm}{Theorem}[section]
\newtheorem{lemma}[thm]{Lemma}
\newtheorem{cor}[thm]{Corollary}
\newtheorem{proposition}[thm]{Proposition}

\theoremstyle{definition}

\theoremstyle{remark}
\newtheorem{exa}[thm]{Example}

\newtheorem*{exa*}{Example}
\newtheorem*{rem*}{Remark}

\DeclareMathOperator{\IN}{\mathbb{N}}
\DeclareMathOperator{\IR}{\mathbb{R}}
\DeclareMathOperator{\IE}{\mathbb{E}}
\DeclareMathOperator{\IP}{\mathbb{P}}

\DeclareMathOperator{\Var}{Var}

\DeclareMathOperator{\Ent}{Ent}

\DeclareMathOperator{\Id}{Id}

\DeclareMathOperator{\diag}{diag}
\DeclareMathOperator{\artanh}{artanh}
\DeclareMathOperator{\Hess}{Hess}
\DeclareMathOperator{\esssup}{\esssup}

\DeclarePairedDelimiter{\abs}{\lvert}{\rvert}
\DeclarePairedDelimiter{\norm}{\lVert}{\rVert}
\DeclarePairedDelimiter{\skal}{\langle}{\rangle}

\newcommand{\eins}{\text{$\mathbbm{1}$}}
\renewcommand{\epsilon}{\varepsilon}
\renewcommand{\phi}{\varphi}
\renewcommand{\tilde}{\widetilde}

\numberwithin{equation}{section}

\begin{document}
\title[Block Spin Ising models]{Fluctuation results for general block spin Ising models}
\author{Holger Kn\"opfel}
\address[Holger Kn\"opfel]{Fachbereich Mathematik und Informatik,
Universit\"at M\"unster,
Einsteinstra\ss e 62,
48149 M\"unster,
Germany}
\email{Holger.Knoepfel@ruhr-uni-bochum.de}

\author{Matthias L\"owe}
\address[Matthias L\"owe]{Fachbereich Mathematik und Informatik,
Universit\"at M\"unster,
Einsteinstra\ss e 62,
48149 M\"unster,
Germany}
\email{maloewe@math.uni-muenster.de}
\thanks{M.L.'s research was funded by the Deutsche Forschungsgemeinschaft (DFG, German Research Foundation) under Germany's Excellence Strategy EXC 2044-390685587, Mathematics M{\"u}nster: Dynamics - Geometry - Structure}

\author{Kristina Schubert}
\address[Kristina Schubert]{Fakult\"at f\"ur Mathematik,
TU Dortmund,
Vogelpothsweg 87,
44227 Dortmund,
Germany}
\email{kristina.schubert@tu-dortmund.de}

\author{Arthur Sinulis}
\address[Arthur Sinulis]{Fakult\"at f\"ur Mathematik,
Universit\"at Bielefeld,
Postfach 100131,
33501 Bielefeld,
Germany}
\email{asinulis@math.uni-bielefeld.de}
\thanks{A.S. acknowledges financial support by the German Research Foundation via the CRC 1283.}

%

\begin{abstract}
	We study a block spin mean-field Ising model, i.\,e. a model of spins in which the vertices are divided into a finite number of blocks with each block having a fixed proportion of vertices, and where pair interactions are given according to their blocks.
	For the vector of block magnetizations we prove Large Deviation Principles and Central Limit Theorems under general assumptions for the block interaction matrix. Using the exchangeable pair approach of Stein's method we establish a rate of convergence in the Central Limit Theorem for the block magnetization vector in the high temperature regime.
\end{abstract}

\maketitle
\section{Introduction}			\label{section:Introduction}
Mean-field block models were introduced as an approximation of a lattice model of a meta-magnet, see e.g.~formula (4.1) in \cite{KC75}. Furthermore, they can arise in disordered systems with random pair interactions, studied for example in \cite{HEC83},\cite{HGHK86},\cite{Com89}. Later, they were rediscovered as interesting models for statistical mechanics systems, see \cite{GC08}, \cite{FC11}, \cite{Col14}, \cite{LS18}, \cite{KT18}, as well as models for social interactions between several groups, e.g.~in \cite{GBC09}, \cite{ABC10}, \cite{OOA18}. This latter approach follows very much the social re-interpretation for {\sl one} group of the Curie-Weiss model in \cite{BD01} or of the Hopfield model in \cite{CL10} or \cite{KL07}. A third source of interest in mean-field spin block models is a statistical point of view. In \cite{BRS17}, the authors gave another analysis of the bipartite mean-field Ising block model with equal block sizes, and asked the question whether one can recover the blocks from several observations from this model, and if so, how many observations are needed.
In this aspect, the block spin models are related to the stochastic block models from random graph theory. These have been in the center of interest in statistics and probability theory over the past couple of years (see, e.g.~\cite{AL18}, \cite{GMZZ17}). The statistical interest in them arises from their relation to graphical models. In this framework a major question is always how to reconstruct the block structure under sparsity assumptions (see e.g.~\cite{BMS13}, \cite{MNS16}, \cite{Bre15}).

Our starting point is \cite{LS18}. There, the fluctuations of an order parameter for a two-groups block model with equal block sizes were analyzed on the level of large deviations principles (LDPs, for short) and central limit theorems (CLTs). Starting from these results, there are several natural questions. First: Can these results be also proven for systems with not necessarily identical block sizes? Second: Can we generalize our results to the situation of more than two groups? And third: Can we give a speed of convergence for the CLT? The main goal of the current note is to (partially) answer these questions. To this end, we will present a new approach to mean-field block spin models, via the corresponding block interaction matrix. Moreover, to obtain a speed of convergence in the CLT, we will employ Stein's method as in \cite{EL10}, \cite{CS11} for the standard mean-field Ising, or Curie--Weiss model.

The rest of this note is organized in the following way. In the remaining part of this introduction, we define our model in a way that makes it accessible to our techniques in Sections 2 and 3, and state our main results. Section 2 is devoted to the proof of the LDP results. Afterwards, we analyze the critical points of the rate function and obtain the \emph{mean field equations}, showing that in the high temperature case the only maximum is $0$, whereas in the low temperature case there are nonzero maximizers, and we obtain a solution for a special class of block interaction matrices. In Section 3 we prove the CLT for the order parameter of the model in two ways. One uses the classical Hubbard--Stratonovich transformation. This was already used for proving the CLT for the magnetization in the Curie--Weiss model in \cite{EN78}, and also is the core technique for the CLT in \cite{LS18}. The second proof uses a multivariate version of the exchangeable pair approach in Stein's method, developed in \cite{RR09}. Lastly, Section \ref{section:Discussion} contains a discussion of some of the results and further open questions.

\subsection{The model}			\label{section:Introduction:themodel}
The block spin Ising model will be characterized by two quantities, a number $k \in \IN$ -- \emph{number of blocks} -- and a symmetric, positive definite matrix $A \in \IR^{k \times k}$, which is the \emph{block interaction matrix}. $A_{ij}$ will determine the strength of interaction between two particles in block $i$ and $j$ respectively. Here, $\IR^{r_1 \times r_2}$ is the set of all $r_1$ by $r_2$ matrices with real entries. \par
Let $N(n)$ be a strictly increasing subsequence of $\IN$. For a system of size $N=N(n)$  let $B_1^{(n)}, \ldots, B_k^{(n)} \subset \{1,\ldots, N\}$ be a partition of $\{1,\ldots, N\}$ into $k$ blocks. Without loss of generality, we assume that the indices in the blocks are ordered, i.e.~if $i_0 \in  B_{i}^{(n)}$ and $j_0 \in B_{j}^{(n)} $ and $i<j$, it follows $i_0 <j_0$.
We call $| B_i^{(n)}|$ the \emph{block size} of the $i$-th block. Note that, in particular, we have a system of size $N$, where for $n \in \mathbb N$
\[ N = N(n) = \sum_{i = 1}^k \abs{B_i^{(n)}}.\]
Define for each $n \in \IN$ the matrix of the relative block sizes
\begin{equation*}
	\Gamma_n \coloneqq \diag \left( \frac{\sqrt{\abs{B_1^{(n)}}}}{\sqrt{N}}, \ldots, \frac{\sqrt{\abs{B_k^{(n)}}}}{\sqrt{N}} \right)  \in \mathbb R^{k \times k}.
\end{equation*}
We assume that for each $i = 1,\ldots, k$ the limit
$$
\gamma_i \coloneqq \lim_{n \to \infty} \sqrt{\frac{\abs{B_i^{(n)}}}{N}} \in (0,1)
$$
exists, so that the matrix of asymptotic relative block sizes
\[
	\Gamma_\infty \coloneqq \diag(\gamma_1, \ldots, \gamma_k) \in \IR^{k \times k}
\]
is invertible. If the $k$ partition blocks are asymptotically of the same size, i.e.
$$
\Gamma_\infty = \frac{1}{\sqrt{k}} \Id \qquad \mbox{ resp.} \quad \gamma_i=\frac{1}{\sqrt{k}} \mbox{ for }i=1, \ldots, k,
$$
we call this the \textit{uniform case}. The block spin Ising model with $k$ blocks of sizes $\abs{B^{(n)}_1}, \ldots, \abs{B^{(n)}_k}$ and block interaction matrix $A$ is defined as the Ising model with interaction matrix
\[
	J_n \coloneqq \frac{1}{N} \begin{pmatrix}
	A_{11} O(\abs{B_1^{(n)}}, \abs{B_1^{(n)}}) & \cdots & A_{1k} O(\abs{B_1^{(n)}}, \abs{B_k^{(n)})} \\
	\vdots & \vdots & \vdots \\
	A_{k,1} O(\abs{B_k^{(n)}}, \abs{B_1^{(n)}}) & \cdots &A_{kk} O(\abs{B_k^{(n)}}, \abs{B_k^{(n)}})
	\end{pmatrix},
\]
where $O(m,n) \in \IR^{m \times n}$ is the matrix with all entries equal to $1$. We denote this model by $\mu_{J_n}$. More precisely, $\mu_{J_n}$ is the probability measure on $\{-1,+1\}^N, N = N(n),$ defined by
\[
	\mu_{J_n}(x) = Z_n^{-1} \exp\left( H_n(x) \right) = Z_n^{-1} \exp\left( \frac{1}{2} \skal{x, J_nx} \right) = Z_n^{-1} \exp\left( \frac{1}{2}\sum_{i,j = 1}^N (J_n)_{ij} x_i x_j \right).
\]
Here, of course, $Z_n$ is the partition function
\[
Z_n := \sum_{ x \in \{-1,+1\}^N } \exp\left( \frac{1}{2}\sum_{i,j = 1}^N (J_n)_{ij} x_i x_j \right).
\]
Note that, contrary to the usual convention, we do not require the diagonal of $J_n$ to be zero for technical convenience. However, since $x_i^2 = 1$, both $J_n$ and its ``dediagonalized'' version $\tilde{J}_n = J_n - \diag(J_{ii})$ give rise to the same Ising model. Here and in the sequel, $\diag(\lambda_1, \ldots, \lambda_l)$ is a diagonal $l \times l$ matrix with values $\lambda_1, \ldots, \lambda_l$ on its diagonal. Lastly, for any $p, q \in [1,\infty]$ and any matrix $A \in \IR^{k \times k}$ we define the operator norm
\[
\norm{A}_{p \to q} \coloneqq \sup_{x \in \IR^k : \norm{x}_p = 1} \norm{Ax}_q.
\]

\subsection{Main results}			\label{section:Introduction:mainresults}
We prove results on the fluctuations of the block magnetization vector on different scales. In what follows, we use the non-normalized and normalized versions of the block magnetization vector defined as
\begin{align*}
	m^{(n)} = m^{(n)}(x) = (m^{(n)}_1(x), \ldots, m^{(n)}_k(x)) &= \left(\sum_{j \in B_i^{(n)}} x_j\right)_{i = 1, \ldots, k}, \\
	\tilde{m}^{(n)} = \tilde{m}^{(n)}(x) = (\tilde{m}^{(n)}_1(x), \ldots, \tilde{m}^{(n)}_k(x)) &= \left( \frac{1}{\abs{B_i^{(n)}}} \sum_{j \in B_i^{(n)}} x_j \right)_{i = 1, \ldots, k},  \\
	\widehat{m}^{(n)} = \widehat{m}^{(n)}(x) = (\widehat{m}^{(n)}_1(x), \ldots, \widehat{m}^{(n)}_k(x)) &= \left( \frac{1}{\sqrt{\abs{B_i^{(n)}}}} \sum_{j \in B_i^{(n)}} x_j \right)_{i = 1,\ldots,k}.
\end{align*}
Note that this allows us to rewrite the Hamiltonian $H_n$ of $\mu_{J_n}$ as
\begin{equation*}
H_n(x) = \frac{1}{2N} \skal*{m^{(n)}, Am^{(n)}} = \frac{1}{2} \skal*{\widehat{m}^{(n)}, \Gamma_n A \Gamma_n \widehat{m}^{(n)}} = \frac{N}{2} \skal*{\Gamma_n^2 A \Gamma_n^2 \tilde{m}^{(n)}, \tilde{m}^{(n)}},
\end{equation*}
which we use tacitly.

We begin by presenting the large deviation results. The first result is a generalization of \cite[Theorem 2.1]{LS18}. In that paper, an LDP for $\tilde{m}^{(n)}$ was proved in the situation of $k=2$ blocks of equal size. Here we analyze the general case.

\begin{thm}	   \label{theorem:LDPgeneralcase}
Let $k \in \IN$ and $A$ be a block interaction matrix. The sequence $(\tilde{m}^{(n)})_{n \in \IN}$ satisfies an LDP under $(\mu_{J_n})_{n \in \IN}$ with speed $N$ and rate function
\begin{equation*}
	J(x) \coloneqq \sup_{y \in \IR^k} I(y) - I(x),
\end{equation*}
where
\begin{equation*}
	I(x) \coloneqq \frac{1}{2} \skal*{x,\Gamma_\infty^2 A \Gamma_\infty^2 x} - \sum_{i = 1}^k \gamma_i^2 L^*(x_i),
\end{equation*}
and $L^*$ denotes the convex conjugate of $\log \cosh$, i.e.
\begin{equation*}
	L^*(x) \coloneqq \frac{1}{2}(1+x) \log(1+x) + \frac{1}{2} (1-x) \log(1-x) \quad \quad x \in [-1,+1].
\end{equation*}
\end{thm}

More precisely, in the notion of large deviations, the sequence of push-forwards $(\tilde{m}^{(n)} \circ \mu_{J_n})_{n \in \IN}$ satisfies an LDP with speed $N$ and the rate function $I$.

In the special case of asymptotically uniform block sizes the function $I$ is related to the matrix $A$ in an even more straightforward way, since in this case
\begin{equation*}
	I(x) = \frac{1}{2k^2} \skal{x, Ax} - \frac{1}{k} \sum_{i = 1}^k L^*(x_i).
\end{equation*}

We show that the rate function $I$ has a unique minimum at $0$ in the case $\norm{\Gamma_\infty^2 A \Gamma_\infty^2}_{2 \to 2} \le 1$, which yields the following corollary.

\begin{cor}\label{corollary:hightemperature}
Under the general assumptions, if $\norm{\Gamma_\infty A \Gamma_\infty}_{2 \to 2} \le 1$, the normalized vector of magnetizations $\tilde{m}^{(n)}$ converges to $0$ exponentially fast in $\mu_{J_n}$-probability. By this we mean more precisely, for each $\varepsilon >0$ there is a constant $I_\varepsilon$ such that 
$$
\mu_{J_n}(||\tilde{m}^{(n)}|| \ge \varepsilon) \le \exp(-N I_\varepsilon).
$$
\end{cor}

Let us discuss the large deviation results. In the classical Curie--Weiss model, i.e.~the case $k=1$, there is a phase transition: The limiting behavior of $\tilde m^{(n)}$ changes, depending on whether $A_{11} \le 1$ (the high temperature regime), or $A_{11} > 1$ (the low temperature regime) (see \cite{Ell06} for an extensive treatment of this model). A corresponding phase transition can be observed in our model. This is stated in \cite{FU12} for the bipartite model. In \cite{KT18} the authors prove the existence of such a phase transition using the method of moments. Of course, with that method one cannot obtain an exponential speed of convergence as in Corollary \ref{corollary:hightemperature}.
In accordance with the notion in the classical Curie--Weiss model, we will call these different parameter regimes the  $\textit{high temperature}$ and  $\textit{low temperature}$ regime, respectively.
Here, the high temperature regime corresponds to $\norm{\Gamma_\infty A \Gamma_\infty}_{2 \to 2} \le 1$ and the low temperature regime to  $\norm{\Gamma_\infty A \Gamma_\infty}_{2 \to 2} > 1$. In the special case of asymptotically uniform block sizes (i.e.~$\Gamma_\infty = \frac{1}{\sqrt{k}} \Id$) these conditions reduce to $\norm{A}_{2 \to 2} \leq k$ and $\norm{A}_{2 \to 2} > k$ respectively.

Next, we consider the scaled block magnetization vector $\widehat{m}^{(n)}$. Again, in the classical (i.e.~one-dimensional) case it is known that the magnetization satisfies a central limit theorem with variance $\sigma^2 = (1-A_{11})^{-1}$ whenever $A_{11} < 1$. The following theorem is a generalization of this phenomenon.

\begin{thm}\label{theorem:CLTgeneralCase}
Let $k \in \IN$ and $A$ be a block interaction matrix. In the high temperature regime we have
\begin{equation*}
	\widehat{m}^{(n)} \Rightarrow \mathcal{N}(0, \Sigma_\infty) = \mathcal{N}\left(0,\left(\Id - \Gamma_\infty A \Gamma_\infty\right)^{-1}\right).
\end{equation*}
Consequently, in the uniform case
\begin{equation*}
	\widehat{m}^{(n)} \Rightarrow \mathcal{N}\left(0,\left(\Id - \frac{1}{k}A\right)^{-1}\right).
\end{equation*}
\end{thm}

\begin{figure}[!ht]
\centering
	\begin{subfigure}[c]{0.49\textwidth}
	\includegraphics[width=\textwidth]{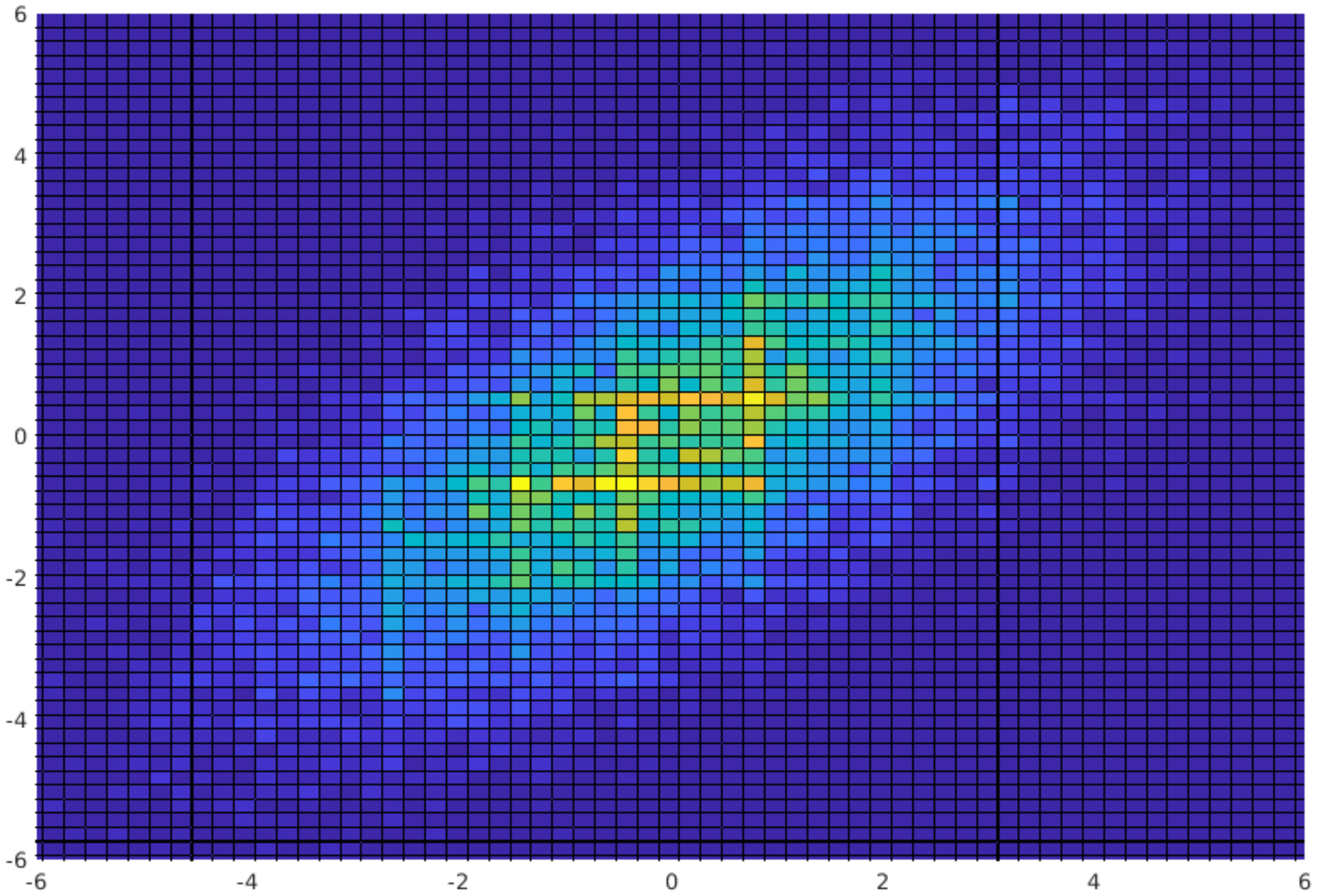}
	\end{subfigure}
	\begin{subfigure}[c]{0.49\textwidth}
	\includegraphics[width=\textwidth]{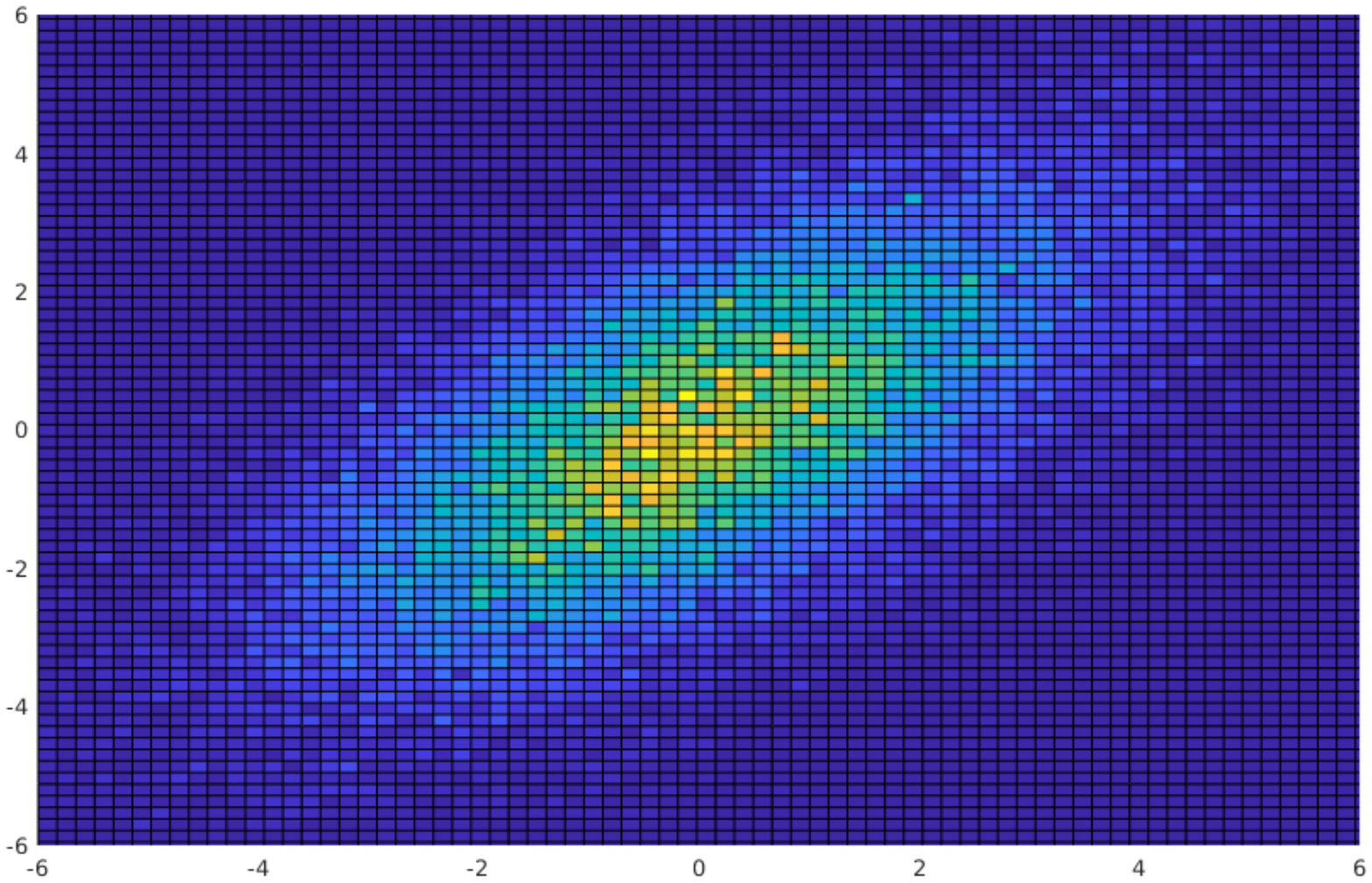}
	\end{subfigure}
  \caption[A visualization of the block magnetization vector, using the Glauber dynamic for sampling, and a heat map for the limiting normal distribution.]{A visualization of the block magnetization vector $\widehat{m}^{(n)}$ (left) for $n = 500$, using the Glauber dynamic for sampling, and a heat map for the limiting normal distribution. Here, we choose $k = 2$, $A = \begin{pmatrix} 1.1 & 0.6 \\ 0.6 & 1.1 \end{pmatrix}$ and the uniform case.}
\end{figure}

Note that $\Sigma_{\infty}$ exists, and it can be expanded into a Neumann series. Moreover, if $\Gamma_\infty A \Gamma_\infty = V^T \Lambda V$ is an orthogonal decomposition, then $\Sigma_\infty = V^T \diag((1 - \lambda_i)^{-1}) V$. Again, a similar statement is derived in \cite{KT18} using the method of moments.

Furthermore, we can treat the critical case. In the Curie--Weiss model, for $\beta = 1$, the quantity $N^{-3/4} \sum_{i = 1}^N \sigma_i$ converges weakly to a measure with Lebesgue-density $g_1(x) \coloneqq Z^{-1} \exp\left( - \frac{x^4}{12} \right)$ (see e.g.~\cite[Theorem V.9.5]{Ell06}). As proven in \cite{LS18} and \cite{FU12} a similar statement holds true for the vector of magnetizations in the case of $k=2$ blocks. The next theorem gives a further generalization of this fact in the case $k \ge 2$. Moreover, it shows that statistics associated to the orthogonal decomposition of the block interaction matrix give rise to $k$ asymptotically independent random variables with either a Gaussian distribution or a distribution with a Lebesgue-density $g_1$. \par
In the multidimensional critical case $\norm{\Gamma_\infty A \Gamma_\infty}_{2 \to 2} = 1$ we restrict to the uniform case with a simple eigenvalue $\lambda_k = k$, i.e.~we have $A = V^T \diag(\lambda_1, \ldots, \lambda_{k-1}, k) V$. Let $	\Gamma_n A \Gamma_n = V^T_n \Lambda_n V_n,$ be the orthogonal decomposition, where $V_n$ is a unitary $k \times k$-matrix and $\Lambda_n$ a diagonal $k\times k$-matrix. If we define the normalized vector
\begin{equation*}
	w' = w'_n \coloneqq  \diag(N^{-1/2}, \ldots, N^{-1/2}, N^{-3/4}) V m^{(n)}
\end{equation*}
and the matrix
\begin{equation*}
	\hat{C}_N \coloneqq \diag(\lambda_1, \ldots, \lambda_{k-1}, kN^{1/2}),
\end{equation*}
we have the following result.

\begin{thm}		\label{theorem:CLTcriticalcase}
Under the above assumptions let $Y_n \sim \mathcal{N}(0, \hat{C}_N^{-1})$ and $X_n \sim \mu_{J_n}$ be independent random variables, defined on a common probability space. Then $w'_n(X_n) + Y_n$ converges in distribution to a probability measure with density
\begin{equation}\label{limdens}
\tilde g_k(x) \coloneqq \tilde{Z}^{-1} \exp\left( - \frac{1}{2} \sum_{i = 1}^{k-1} \left( \lambda_i - \frac{\lambda_i^2}{k} \right) x_i^2 - \frac{k^3}{12} x_k^4 \sum_{i = 1}^k V_{ki}^4 \right)
\end{equation}
for a suitable normalization $\tilde Z$ that makes the expression \eqref{limdens} a probability density.\par 
Thus, the vector $(w_n'(X_n)_j)_{j = 1,\ldots, k-1}$ converges to a normal distribution with covariance matrix $\Sigma = \diag\left( (k-\lambda_j)^{-1}\right)$ and the random variable $w_n'(X_n)_k$ converges to a distribution with Lebesgue-density $Z^{-1} \exp \left( - (\frac{k^3}{12} \sum_{i = 1}^k V_{ki}^4) x^4 \right) dx$.
\end{thm}

We believe it is possible to extend Theorem \ref{theorem:CLTcriticalcase} to the case where the eigenvalue $k$ has multiplicity greater than $1$, by appropriately rescaling all the eigenvectors which belong to the eigenvalue $k$. 

Note that the parameter $\sigma^2 \coloneqq k^3/12 \sum_{i = 1}^k V_{ki}^4$ is directly related to the variance of a random variable with that distribution; indeed, a short calculation shows that for $X \sim \exp(-\sigma^2 x^4) dx$ we have $\mathrm{Var}(X) = c \sigma^{-1}$, where $c$ is an absolute constant. Moreover, $\sum_{i = 1}^k V_{ki}^4 = \norm{v_k}_4^4$, where $v_k$ is the eigenvector belonging to the eigenvalue~$k$.

In a final step, we establish convergence rates in the CLT in the high temperature case for a special class of functions. We use the exchangeable pair approach of Stein's method, that was also used in \cite{EL10} and \cite{CS11} in the case of the Curie--Weiss model. The proof of the next result will rely on a multivariate version of Stein's method proven in \cite{RR09}. To this end, define the function class
\begin{equation*}
\mathcal{F}_3 \coloneqq \left \lbrace h : \IR^k \to \IR : h \in \mathcal{C}^3(\IR^k), \max_{j = 1,2,3} \max_{\alpha = (\alpha_j)_j} \sup_{x \in \IR^k} \abs{\partial^\alpha h}(x) \le 1 \right \rbrace
\end{equation*}
of all three times differentiable functions with all partial derivatives (up to order three) bounded.

\begin{thm}			\label{theorem:CLTviaSteinsMethod}
Assume that $\norm{\Gamma_\infty A \Gamma_\infty}_{2 \to 2} < 1$ and for each $n \in \IN$ let $\Sigma_n\coloneqq \IE \hat{m}^{(n)} (\hat{m}^{(n)})^T$. For $Z \sim \mathcal{N}(0,\Id)$, we have
\begin{equation*}
\sup_{h \in \mathcal{F}_3} \big \lvert\IE_{\mu_{J_n}}\big( h\big(\hat{m}^{(n)}\big) \big) - \IE h(\Sigma_n^{1/2} Z) \big \rvert = O(N^{-1/2}).
\end{equation*}
\end{thm}

\section{Proofs of the large deviation results and the mean-field equations}	\label{section:ProofsLDP}
Let us start off by proving the LDP result for the rescaled block magnetization vector $\tilde{m}^{(n)}$.
Recall the notion of an LDP (for which we also refer to \cite{dH00} and \cite{DZ10}): If $\mathcal{X}$ is a Polish space and $(a_n)_{n \in \IN}$ is an increasing sequence of non-negative real numbers, we say that a sequence of probability measures $(\nu_n)_n$ on $\mathcal{X}$ satisfies a \emph{large deviation principle} with speed $a_n$ and rate function $I: \mathcal{X} \to \IR$ (i.e.~a lower semi-continuous function with compact level sets $\{x: I(x) \le L\}$ for all $L>0$), if for all Borel sets $B \in \mathcal{B}(\mathcal{X})$ we have
\begin{equation*}
	-\inf_{x \in \mathrm{int}(B)} I(x) \le \liminf_{n \to \infty} \frac{\log \nu_n(B)}{a_n} \le \limsup_{n \to \infty} \frac{\log \nu_n(B)}{a_n} \le -\inf_{x \in \mathrm{cl}(B)} I(x),
\end{equation*}
where $\mathrm{int}(B)$ and $\mathrm{cl}(B)$ denote the topological interior and closure of a set $B$, respectively.

We say that a sequence of random variables $X_n: \Omega \to \mathcal{X}$ satisfies an LDP with speed $a_n$ and rate function $I: \mathcal{X} \to \IR$ under a sequence of measures $\mu_n$ if the push-forward sequence $\nu_n \coloneqq \mu_n \circ X_n$ satisfies an LDP with speed $a_n$ and rate function $I$. \par

To prove Theorem \ref{theorem:LDPgeneralcase}, we will need the following lemma.

\begin{lemma}		\label{lemma:LDPwithSmallPerturbation}
	Let $\mathcal{X}$ be a Polish space and assume that a sequence of measures $(\mu_n)_{n \in \IN}$ on $\mathcal{X}$ satisfies an LDP with speed $n$ and rate function $I$. Let $F: \mathcal{X} \to \IR$ be a continuous function which is bounded from above and $\eta_n: \mathcal{X} \to \IR$ a sequence of functions such that $\norm{\eta_n}_{L^\infty(\mu_n)} \to 0$. Then the sequence of measures
	\[
		d\tilde{\mu}_n= \exp(nF + n\eta_n) d\mu_n
	\]
	satisfies an LDP with speed $n$ and rate function
	\[
		J(x) = \sup_{\lambda \in \mathcal{X}} \left( F(\lambda) - I(\lambda) \right) - (F(x) - I(x)).
	\]
\end{lemma}

\begin{proof}
Note that this is a slight modification of the {\it tilted LDP}, which is an immediate consequence of Varadhan's Lemma (\cite[Theorem III.17]{dH00}). Indeed, according to this tilted LDP, the sequence of measures $(\nu_n)_n$ with $\mu_n$-density $\exp(nF)$ satisfies an LDP with speed $n$ and rate function $J$. Since for any $n \in \IN$ and any $B \in \mathcal{B}(\mathcal{X})$ the inequalities
	\[
		e^{-2n\norm{\eta_n}_{L^{\infty}(\mu_n)}} \nu_n(B) \le \tilde{\mu}_n(B) \le e^{2n \norm{\eta_n}_{L^\infty(\mu_n)}} \nu_n(B)
	\]
hold, this easily implies an LDP for $(\tilde{\mu}_n)_n$ with speed $n$ and the same rate function $J$ due to $\norm{\eta_n}_{L^\infty(\mu_n)} \to 0$.
\end{proof}

\begin{proof}[Proof of Theorem \ref{theorem:LDPgeneralcase}]
First, note that under the uniform measure $\mu_0$ (i.e.~$A \equiv 0$) we have
\[
	\IE_{\mu_0} \exp\left( N \skal{t,\tilde{m}^{(n)}} \right) = \prod_{i = 1}^k \cosh\left(t_i \frac{N}{\abs{B_i^{(n)}}}\right)^{\abs{B_i^{(n)}}},
\]
so that
\[
	\lim_{N \to \infty} \frac{1}{N} \log \IE_{\mu_0} \exp\left( N \skal{t,\tilde{m}^{(n)}} \right) = \sum_{i = 1}^k \gamma_i^2 \log \cosh \left( \frac{t_i}{\gamma_i^2} \right).
\]
By the G\"artner-Ellis Theorem (\cite[Theorem 2.3.6]{DZ10}), $\tilde{m}^{(n)}$ satisfies an LDP under $\mu_0$ with speed $N$ and rate function
\[
	J_{\mu_0}(x) \coloneqq \sup_{t \in \IR^k} \left( \skal{t,x} - \sum_{i = 1}^k \gamma_i^2 \log \cosh\left( \frac{t_i}{\gamma_i^2} \right) \right) = \sum_{i = 1}^k \gamma_i^2 L^*(x_i),
\]
where $L^*(x)$
is the convex conjugate of $\log \cosh$. Next, it is easy to see that we can rewrite the $\mu_0$-density of $\mu_{J_n}$ as
\[
	\frac{d\mu_{J_n}}{d\mu_0}(x) = \exp \left( \frac{N}{2} \skal{(\Gamma_n^2 A \Gamma_n^2)\tilde{m}^{(n)},\tilde{m}^{(n)}} \right) = \exp \left( N F(\tilde{m}^{(n)}) - N \eta_n(\tilde{m}^{(n)}) \right),
\]
where
\begin{align*}
	F(x) &= \frac{1}{2} \skal{\Gamma_\infty^2 A \Gamma_\infty^2 x, x} = \frac{1}{2} \skal{\Gamma_\infty^2 A \Gamma_\infty^2 x, x} \wedge \frac{1}{2} k\norm{\Gamma_\infty^2 A \Gamma_\infty^2}_{2 \to 2}, \\
	\eta_n(x) &= \frac{1}{2} \skal{\Gamma_\infty^2 A \Gamma_\infty^2 x,x} - \frac{1}{2} \skal{\Gamma_n^2 A \Gamma_n^2 x,x}.
\end{align*}
Note that we artificially inserted the truncation in $F$ to emphasize the boundedness of $F(\tilde{m}^{(n)})$. This does not affect the quadratic form, as
\[
\left|\frac{1}{2} \skal*{\Gamma_\infty^2 A \Gamma_\infty^2 \tilde{m}^{(n)}, \tilde{m}^{(n)}}\right| \leq \frac{1}{2} \norm{\Gamma_\infty^2 A \Gamma_\infty^2}_{2 \to 2} \norm{\tilde{m}^{(n)}}_2^2 \leq \frac{k}{2} \norm{\Gamma_\infty^2 A \Gamma_\infty^2}_{2 \to 2}.
\]
Moreover, $F$ is obviously continuous and $\eta_n$ satisfies
\begin{align*}
	\norm{\eta_n}_\infty \le k \norm{\Gamma_\infty^2 A \Gamma_\infty^2 - \Gamma_n^2 A \Gamma_n^2} \to 0
\end{align*}
on the support of $\mu_0 \circ \tilde{m} = [-1,1]^k$, so that the assertion follows from Lemma \ref{lemma:LDPwithSmallPerturbation}.
\end{proof}

\subsection{The mean-field equations}
Theorem \ref{theorem:LDPgeneralcase} states that the function
\[
I(x) = \frac{1}{2} \skal*{x, \Gamma_\infty^2 A \Gamma_\infty^2 x} - \sum_{i = 1}^k \gamma_i^2 L^*(x_i)
\]
determines the asymptotic behavior of the magnetization, and thus the critical points of $I$ are of utter importance. These satisfy the so-called \emph{mean-field equations}
\begin{align}		\label{eqn:meanfield}
\begin{split}
x_1 &= \tanh((A \Gamma_\infty^2 x)_1) = \tanh\bigg(\sum_{j = 1}^k A_{1j} \gamma_j^2 x_j\bigg) \\
\vdots &\quad \vdots \quad \vdots \\
x_k &= \tanh((A \Gamma_\infty^2 x)_k) = \tanh\bigg(\sum_{j = 1}^k A_{kj} \gamma_j^2 x_j\bigg).
\end{split}
\end{align}
For example, in the well-studied case $k = 2$, choosing 
\[ 
	A = \begin{pmatrix} A_{11} & A_{12} \\ A_{12} & A_{22} \end{pmatrix} \quad \text{ and }\quad \Gamma_\infty^2 = \begin{pmatrix} \gamma & 0 \\ 0 & 1-\gamma \end{pmatrix}
\] 
for a positive definite matrix $A$ and $\gamma \in (0,1)$ equations   \eqref{eqn:meanfield} reduce to
\begin{align*}		 
\begin{split}
x_1 &= \tanh(\gamma A_{11} x_1 + (1-\gamma) A_{12} x_2), \\
x_2 &= \tanh(\gamma A_{12} x_1 + (1-\gamma) A_{22} x_2).
\end{split}
\end{align*}
Whereas for the two-dimensional fixed point problem the existence of a solution can be shown by monotonicity arguments, the existence of a solution to  \eqref{eqn:meanfield} for general $k$ is more involved. First off, we show that in the high temperature regime the only critical point of $I$ is $0$. This will immediately yield Corollary \ref{corollary:hightemperature}.

\begin{proof}[Proof of Corollary \ref{corollary:hightemperature}]
In the sense of the formulation in Corollary \ref{corollary:hightemperature}, $\tilde{m}^{(n)}$ concentrates exponentially fast in the minima of the function $J$. However, under the condition $\norm{\Gamma_\infty A \Gamma_\infty}_{2 \to 2} \le 1$ there is only one minimum, which is zero. To see this, note that any local minimum satisfies
\begin{align}	\label{eqn:nablazero}
	\nabla(J)(x) = - \Gamma_\infty^2 A \Gamma_\infty^2 x + \Gamma_\infty^2 \artanh(x) = 0.
\end{align}
Here, $\artanh(x)$ is understood componentwise. Clearly, $0$ is a solution, and due to
\begin{equation}\label{eqn:HessianInZero}
	\Hess\left(I\right)(0) = -\Gamma_\infty^2 A \Gamma_\infty^2 + \Gamma_\infty^2 = \Gamma_\infty \left( \Id - \Gamma_\infty A \Gamma_\infty \right) \Gamma_\infty \ge 0
\end{equation}
this is a local minimum. Assume there is some $y \neq 0$ solving \eqref{eqn:nablazero}, and observe that
\begin{align}\label{critpointdisc}
\norm{\Gamma_\infty^2 y}_2^2 &= \skal*{\Gamma_\infty^2 y, \Gamma_\infty^2 \artanh(y) + \left(\Id - \Gamma_\infty^2 A \right) \Gamma_\infty^2 y} \ge \skal{\Gamma_\infty^2 y, \Gamma_\infty^2 \artanh(y)} \nonumber \\
&= \sum_{i = 1}^k \gamma_i^4 \artanh(y_i) y_i \ge \sum_{i = 1}^k \gamma_i^4 y_i^2 = \norm{\Gamma_\infty^2 y}_2^2.
\end{align}
Here the first inequality follows from the general fact that the spectrum of the matrices $BC$ and $CB$ agree, applied to $B = \Gamma_\infty$ and $C = \Gamma_\infty A$. The last inequality follows from $\artanh(x)x \ge x^2$ for all $x \in (-1,1)$, with equality for $x = 0$ only. This means that for any solution $y$ we have equality in \eqref{critpointdisc}. However, equality can only hold if $y_i = 0$ whenever $\gamma_i \neq 0$. Due to our assumption $\gamma_i \in (0,1)$, this proves the claim.
\end{proof}

In contrast, in the low temperature regime, there are other solutions to the mean-field equations \eqref{eqn:meanfield}. Let us start with the following proposition showing the connection of the $k$-dimensional mean-field equations to the one-dimensional equations of the Curie--Weiss model. It provides an explicit formula for the solution of the $k$-dimensional problem in terms of the solution of the Curie--Weiss equation.

\begin{proposition}\label{lemma:solutionToMFE}
Let $k \in \IN$, $\Gamma_\infty = \frac{1}{\sqrt{k}} \Id$ and  $A$ be a positive semidefinite, symmetric matrix with $\norm{A}_{2 \to 2} > k$. If the eigenvector $v_k$ belonging to the largest eigenvalue $\lambda_k$ can be rescaled to satisfy $v_k \in \{-1,0,1\}^k$, then there exists a solution $x \neq 0$ to the mean-field equations \eqref{eqn:meanfield} and it is given by $x = m^* v_k$, where $m^*$ is the positive one-dimensional solution of the Curie--Weiss model with temperature $\beta = \lambda_k k^{-1} > 1$.
\end{proposition}

\begin{proof}
Let $m^* > 0$ be the unique positive solution of the Curie--Weiss equation $\tanh(\frac{\lambda_k}{k} x) = x$ for $\beta \coloneqq \frac{\lambda_k}{k} > 1$ and define $v \coloneqq m^* v_k$. We have
\[
	\tanh\bigg(\frac{1}{k} A v\bigg) = \tanh\bigg(\frac{m^*}{k} A v_k \bigg) = \tanh\bigg( \frac{m^* \lambda_k}{k} v_k \bigg) = \tanh\bigg( \frac{m^* \lambda_k}{k} \bigg) v_k = v,
\]
where in the second-to-last step we have used explicitly that $v_k \in \{ -1,0,1 \}^k$, and so $v$ is a critical point of $I$. Moreover, in this case it is easily seen that
\[
	\Hess\left( \frac{1}{2k^2} \skal{x,Ax} - \frac{1}{k} \sum_{i = 1}^k L^*(x_i) \right)(v) = \frac{1}{k^2}A - \frac{1}{k(1-(m^*)^2)} \Id
\]
is negative definite. Indeed, from
\[
	g(x) \coloneqq \frac{\artanh(x)}{x} - \frac{1}{1-x^2} = \sum_{k = 0}^\infty x^{2k}\left( \frac{1}{1+2k} - 1 \right) \le 0
\]
we obtain
\begin{align*}
	\skal*{y,\left( \frac{1}{k^2} \Lambda - \frac{1}{k(1-(m^*)^2)}\right)y } &= \frac{1}{k} \sum_{i = 1}^k y_i^2 \left( \frac{\lambda_i}{k} - \frac{1}{1-(m^*)^2} \right) \\ &\le \frac{1}{k} \sum_{i = 1}^k y_i^2 \left( \frac{\lambda_k}{k} - \frac{1}{1-(m^*)^2} \right) = \frac{g(m^*)}{k} \sum_{i = 1}^k y_i^2 \\ &\le 0.
\end{align*}
\end{proof}

\begin{exa}
Even though the assumptions in the previous proposition seem to be tailor-made for its proof (and the conclusion also holds true more generally), there are interesting non-trivial examples of a matrix satisfying the conditions of Proposition \ref{lemma:solutionToMFE}. One of them is the family of $k \times k$ matrices ($k \in \IN$) of the form
\begin{equation*}
  A(\alpha, \beta) = (\beta - \alpha) \Id + \alpha O(k,k)
\end{equation*}
for any parameters $(\alpha, \beta)$ satisfying
\begin{align}\label{eqn:condExampleMatrix}
\beta + (k-1) \alpha > k \quad \text{and} \quad \beta > \alpha.
\end{align}
This corresponds to $k$ groups with an interaction parameter $\beta$ within the group and $\alpha$ between the groups. For example, the condition \eqref{eqn:condExampleMatrix} is satisfied whenever $\beta > \alpha > 1$.
\end{exa}

In the general case, the conclusion of Proposition \ref{lemma:solutionToMFE} holds as well. In this case the proof relies on the fact that the continuous function $I$ has a global maximum on its (compact) domain $[-1,1]^k$, and the next lemma excludes maxima on the boundary. Hence there is always at least one solution $y \neq 0$ (since $0$ is either an inflection point or a minimum) to \eqref{eqn:meanfield}.

\begin{lemma}\label{lemma:prop_H}
Let $I$ be the large deviation rate function from Theorem \ref{theorem:LDPgeneralcase}, i.e. 
\[
	I(x) \coloneqq \frac{1}{2} \skal*{x,\Gamma_\infty^2 A \Gamma_\infty^2 x} - \sum_{i = 1}^k \gamma_i^2 L^*(x_i), \quad x \in[-1,1]^k
\]
and $L^*$ denotes the convex conjugate of $\log \cosh$. 
\begin{enumerate}
\item $I$ has no global maxima on the boundary of $[-1,1]^k$.
\item If $x \in [-1,1]^k$ satisfies the mean-field equations, we have
\begin{equation}	\label{eqn:DarstellungHimMaximum}
	I(x) = \frac{1}{2} \sum_{i = 1}^k \gamma_i^2 \left( x_i \artanh(x_i) + \log(1-x_i^2) \right).
\end{equation}
\item The set of all global maximisers has a positive distance from the boundary.
\end{enumerate}
\end{lemma}

\begin{proof}
$(1)$: Assume that $x$  is a global maximum of $I$ on the boundary. Then there is at least one index $j \in \{1,\ldots,k\}$ such that $x_j = 1$ (if $x_j = -1$, switch to $-x$ since $I(-x) = I(x)$). Rewriting the fact that $x$ is a maximum of $I$, we have for any $y_j \in [-1,1]$ and $C \coloneqq \Gamma_\infty^2 A \Gamma_\infty^2$
\begin{equation*}
\frac{1}{2} \left( \skal*{(\overline{x}_j,1), C(\overline{x}_j,1)} - \skal*{(\overline{x}_j,y_j), C (\overline{x}_j,y_j)} \right) \ge \gamma_j^2 (L^*(1) - L^*(y_j)),
\end{equation*}
where $\overline{x}_j\in\mathbb R^{k-1}$ is the vector obtained from $x$ by deleting the $j$-th component.
If we divide both sides by $1-y$ and let $\limsup_{y \to 1}$, the left hand side is finite, as $\frac{1}{2} \skal{x,Cx} \in C^\infty(\IR^k)$, and the right hand side tends to $\infty$ by l'Hospital's rule. This proves statement (1). \par
$(2)$: Clearly, $x$ can only satisfy the mean-field equations if $x \in (-1,+1)^k$. Since it solves the mean-field equations, for any $i = 1,\ldots,k$ we have
\begin{align*}
\artanh(x_i) = \sum_{j = 1}^k A_{ij} \gamma_j^2 x_j = (A \Gamma_\infty^2 x)_i.
\end{align*}
Inserting this into the function $I$ gives
\begin{align*}
I(x) &= \frac{1}{2} \sum_{i = 1}^k \gamma_i^2 x_i (\Gamma_\infty^2 A x)_i - \sum_{i = 1}^k \gamma_i^2 L^*(x_i) = \frac{1}{2} \sum_{i = 1}^k \gamma_i^2 (x_i \artanh(x_i) - 2L^*(x_i)) \\
&= - \frac{1}{2} \sum_{i = 1}^k \gamma_i^2 (x_i \artanh(x_i) + \log(1-x_i^2)) \\
&\eqqcolon \frac{1}{2} \sum_{i = 1}^k \gamma_i^2 R(x_i).
\end{align*}
$(3)$: The function $I$ is bounded in $[-1,1]^k$, as
\[
\abs{I(x)} \le 2 \log 2 + \frac{k}{2} \norm{\Gamma_\infty^2 A \Gamma_\infty^2}_{2 \to 2}.
\]
On the other hand, if there exists a sequence of maximisers approaching the boundary, i.e.~for at least one $i$ we have $x_i \to 1$, this gives $R(x_i) \to \infty$.
\end{proof}

In the case of two blocks, i.e.~$k=2$, equal block sizes and the same interaction within a group, the set of maximisers of the rate function is explicitly known. Indeed, in \cite[Proposition 4.1]{BRS17} and \cite[Theorem 2.1]{LS18} the authors show that for 
\[
A = \begin{pmatrix} \beta & \alpha \\ \alpha & \beta \end{pmatrix}
\]
satisfying $\beta \ge \alpha \ge 0$ and $\beta + \alpha > 2$ (the low temperature case) the distribution of $\tilde{m}^{(n)}$ concentrates in the two points $x = (m^+((\beta+\alpha)/2), m^+((\beta+\alpha)/2)$, and $-x$. In the case $\alpha < 0$ the limit points for $\tilde{m}^{(n)}$ become
$x = (m^+((\beta+\alpha)/2), -m^+((\beta+\alpha)/2))$, and $-x$. Here $m^+(b)$ is the largest solution to
$$
m=\tanh(bm).
$$
If $\beta + \abs{\alpha} \le 2$, the distribution of $\tilde{m}^{(n)}$ concentrates in the origin. 
For $k=2$, we can extend this result to arbitrary block sizes. 

\begin{proposition}\label{lemma:2BlocksNumberOfMaxima}
Let $k = 2$, $A = \begin{pmatrix} \beta & \alpha \\ \alpha & \beta \end{pmatrix}$ be a block interaction matrix and $\gamma_1^2=\gamma$ $\gamma_2^2 =(1-\gamma)$ for some $0 <\gamma< \frac{1}{2}$. In the low temperature case, if the groups are not interacting (i.e.\, $\alpha = 0$) there exists either two or four global maxima of $I$; for $\alpha \neq 0$, there are always two global maxima of $I$.
\end{proposition}

Note that we have to restrict to $\abs{\alpha} < \beta$ and $\beta > 0$ in order for $A$ to be positive definite. Moreover, the characterization of the high temperature phase $\Gamma_\infty A \Gamma_\infty \preceq \mathrm{Id}$ (where $\preceq$ is the Loewner partial ordering) can be reduced to $\skal{(\mathrm{Id} - \Gamma_\infty A \Gamma_\infty)e_1, e_1} > 0$ and $\det(\mathrm{Id} - \Gamma_\infty A \Gamma_\infty) > 0$. Thus we are in the high temperature regime if and only if
\[
\beta \gamma < 1 \quad \text{and} \quad (\beta^2 - \alpha^2) \gamma(1-\gamma) > \beta - 1.
\]

\begin{proof}
The case $\alpha = 0$ is an easy consequence of the statements for the one-dimensional Curie--Weiss model, since $I(x_1,x_2) = I_1(x_1) + I_2(x_2)$ and $\Gamma_\infty^2 A \Gamma_\infty^2 = \mathrm{diag}(\beta \gamma^2, \beta (1-\gamma^2))$. 

We treat the case $\alpha > 0$ only -- the case $\alpha < 0$ follows immediately from the equality $I_{\alpha,\beta}(x,y) = I_{-\alpha,\beta}(x,-y)$ (with the appropriate modifications, e.g.~the maximum will be in the second quadrant instead of the first). 

Due to \eqref{eqn:DarstellungHimMaximum} the maximum of the rate function is non-negative, let us call this maximum $\eta$. Then,  $I(x,y)=\eta = 0$ implies $(x,y) = 0$, which is a contradiction to the low temperature case (recall the Hessian of $I$ in $0$ given in equation \eqref{eqn:HessianInZero}), so that $\eta > 0$. Moreover, every global maximum (and thus local maximum, as it is not attained on the boundary) satisfies the mean-field equations, and so the value of $I$ at any maximum is given by equation \eqref{eqn:DarstellungHimMaximum}. As a consequence, all global maxima lie on a contour line $C_\eta \coloneqq \{ x \in [-1,1]^2 : \gamma R(x_1) + (1- \gamma) R(x_2) = 2\eta \}$, where $R(x) = x \artanh(x) + \log(1-x^2)$ was defined in the previous lemma.

Firstly, let us show that in the first quadrant there can only be one such point. Due to symmetry, the global maximum will also be present in the third quadrant. For $x_1 > 0$ the points on the contour line $C_{\eta}$ can be described by a function $x_2 = g(x_1)$, and due to the monotonicity of $R$ the function $g$ is non-increasing. Moreover, the solutions of the mean-field equations can be described by the functions
\begin{align*}
\begin{split}
f_1(x)&:=\frac{1}{(1-\gamma)\alpha} \left(\artanh(x)- \gamma \beta x \right), \\
f_2(x)&:= \frac{1}{\gamma \alpha}\left(\artanh(x)- (1-\gamma) \beta x \right)
\end{split}
\end{align*}
via 
$$x_2=f_1(x_1) \quad \text{and} \quad x_1= f_2(x_2).$$
The function $f_1$ can behave in two ways, depending on the parameter $\gamma \beta$: For $\gamma\beta \le 1$ it increases monotonously. For $\gamma\beta > 1$  it decreases first and then increases. More precisely, in the latter case, $f_1(t)=0$ if and only if $t \in \{0, \pm m_{\gamma \beta}\}$ for some $m_{\gamma \beta} > 0$ and $f_1$ is strictly increasing for $t \ge m_{\gamma \beta}$. Moreover, the curve $(x, f_1(x))$ is only in the first quadrant if $m_{\gamma \beta} <x \leq 1$. 
In either case, there is only one intersection point of $g$ and $f_1$ in the first quadrant. \par
Secondly, the maximum cannot be in the second quadrant. Assume that there are solutions to the mean field equations both in the first and in the second quadrant. If we denote by $m_c$ the zeros of $\phi_c(t) \coloneqq \artanh(x) - cx$, for the solution in the second quadrant, we easily see  that  $-m_c<x <0$ and $0\leq y \leq m_{\beta (1- \gamma)}.$ Hence
\begin{equation*}
 I(x,y)=\frac{1}{2} \left( \gamma R(x) + (1-\gamma) R(y) \right) < \frac{1}{2} \left( \gamma R(m_{\beta \gamma}) + (1-\gamma) R(m_{(1-\gamma) \beta}) \right).
\end{equation*}
If there is also a solution in the first quadrant with coordinates $(x^*, y^*)$, we obtain analogously 
\begin{equation*}
 I(x^*,y^*)> \frac{1}{2} \left( \gamma R(m_{\beta \gamma)} + (1-\gamma) R(m_{(1-\gamma) \beta}) \right).
\end{equation*}
This yields that the maximum must lie in the first quadrant 
\end{proof}

Furthermore, we can treat the case $k > 2$ for uniform block sizes and special matrices. The proof is motivated by \cite[Proposition 4.1]{BRS17}.

\begin{lemma}\label{lemma:BlockInteractionTwomaximisers}
Let $k \geq 2$ and $A$ be a block interaction matrix with positive entries such that we have for any $i = 1,\ldots, k$ for two constants $c_1, c_2 > 0$ $A_{ii} = c_1$ and $\sum_{j \neq i} A_{ij} = c_2$.

In the uniform case, there are exactly two maximisers of the rate function $I$ and they satisfy $x = m^* (1,\ldots, 1)$ for $m^*$ solving the Curie--Weiss equation $\frac{c_1 + c_2}{k} x = \artanh(x)$.
\end{lemma}

\begin{proof}
Using the equality $xy = -\frac{1}{2} (x-y)^2 + \frac{1}{2} x^2 + \frac{1}{2} y^2$ we can rewrite the rate function as
\begin{align*}
I(x) &= \frac{1}{k} \left( \frac{1}{2k} \sum_{i \neq j} A_{ij} x_i x_j + \frac{c_1}{2k} \skal{x,x} - \sum_{i = 1}^k L^*(x_i) \right) \\
&= \frac{1}{k} \left( - \frac{1}{4k} \sum_{i,j} A_{ij} (x_i - x_j)^2 + \frac{c_1 + c_2}{2k} \skal{x,x} - \sum_{i = 1}^k L^*(x_i) \right) \\
&\le \frac{c_1 + c_2}{2k^2} \skal{x,x} - \frac{1}{k} \sum_{i = 1}^k L^*(x_i), 
\end{align*}
where equality only holds in the case $x_i = x_j$ for all $i,j$. Thus, we search for maximisers of $I$ on the generalized diagonal $\{ x \in [-1,1]^k : x_i = x_j\, \forall i,j \}$. On this set we have
\[
I((x,\ldots,x)) = \frac{c_1 + c_2}{2k} x^2 - L^*(x),
\]
i.e it reduces to the Curie--Weiss equations in one dimension. For $c_1 + c_2 > k$ it has a unique nonzero solution $m^*$, and $x = m^* (1,\ldots, 1)$ solves the $k$-dimensional maximization problem.
\end{proof}

Unfortunately, the proof cannot be modified in a straightforward way to deal with non-equal block sizes, not even in the case $k = 2$. The reason is that the inequality used in the proof does not give any information on the actual maximiser in this setting (i.\,e. $I$ is not maximized on any type of (weighted) diagonal). As such, we cannot reduce this to the one-dimensional setting.

\begin{exa*}
For example, Lemma \ref{lemma:BlockInteractionTwomaximisers} can be used to prove that given three positive parameters $\alpha,\beta,\gamma$ with $\beta > \alpha$ and $\beta + \alpha > 2 \gamma$, the rate function corresponding to
\[
A = \begin{pmatrix}
\beta & \alpha & \gamma & \gamma \\
\alpha & \beta & \gamma & \gamma \\
\gamma & \gamma & \beta & \alpha \\
\gamma & \gamma & \alpha & \beta
\end{pmatrix}
\]
only has two maximisers in the uniform case. The conditions on $\alpha, \beta, \gamma$ ensure that $A$ is positive definite, and it is clear that $c_1 = \beta$ and $c_2 = \alpha + 2 \gamma$.
\end{exa*}

As a concluding remark let us note that the previous results imply that there is indeed a phase transition in our block spin model. 
However, if $k>2$ or the block sizes are not equal, it seems hard to give a similarly explicit formula for the limit points.
Nevertheless, the above observations show that there is a phase transition in a very general class of block spin  models with an arbitrary number of blocks and general class of block sizes. In particular, they also justify the names ``high temperature regime'' and ``low temperature regime''.

\section{Proofs of the limit theorems}		\label{section:ProofsLimitTheorems}
In this section we prove (standard and non-standard) Central Limit Theorems for the vector $\widehat{m}^{(n)}$. In the first subsection we will treat the high temperature regime. Here we derive a standard CLT using the Hubbard--Stratonovich transform. This is in spirit similar to the third section in \cite{LS18} and technically related to \cite{GL99}. The result can also be derived from \cite{FC11}, where similar techniques are used. However, the subsection also prepares nicely for Subsection 3.2, where we treat the critical case and show a non standard CLT. This generalizes results from \cite{FU12} and \cite{LS18}. Finally, in Subsection 3.3 we will use Stein's method, an alternative approach to prove the CLT for $\widehat{m}^{(n)}$. This is not only interesting in its own right, but also has the advantage of providing a speed of convergence, which is missing in the case of a proof via the Hubbard--Stratonovich transform.

\subsection{Central limit theorem: Hubbard--Stratonovich approach}
For the proof we shall use the transformed block magnetization vectors
\begin{align*}
	w^{(n)} &\coloneqq V_n m^{(n)},	\\
	\widehat{w}^{(n)} &\coloneqq V_n \widehat{m}^{(n)}, \\
	\tilde{w}^{(n)} &= V_n \Gamma_n \tilde{m}^{(n)},
\end{align*}
where $\Gamma_n A \Gamma_n = V_n^T \Lambda_n V_n$ is the orthogonal decomposition. It is easy to see that
\begin{equation*}
	H_n = \frac{1}{2N} \skal{w^{(n)}, \Lambda_n w^{(n)}} = \frac{1}{2} \skal*{\widehat{w}^{(n)}, \Lambda_n \widehat{w}^{(n)}} = \frac{N}{2} \skal*{\Lambda_n \tilde{w}^{(n)}, \tilde{w}^{(n)}}.
\end{equation*}
\begin{proof}[Proof of Theorem \ref{theorem:CLTgeneralCase}]
As in \cite{LS18} or \cite{FC11} (both papers are inspired by \cite{EN78}), we use the Hubbard--Stratonovich transform (i.e.~a convolution with an independent normal distribution). For each $n \in \IN$, 
\begin{equation*}
	\mu_{J_n}(\sigma) =  Z_n^{-1} \exp \left( \frac{1}{2} \skal{\Lambda_n \widehat{w}^n, \widehat{w}^n} \right).
\end{equation*}
Our first step is to prove that $\widehat{w}^n$ converges weakly to a normal distribution. Let $Y_n \sim \mathcal{N}(0, \Lambda_n^{-1})$ be an independent sequence, which is moreover independent of $(\widehat{w}^n)_{n \in \IN}$. We have for any $B \in \mathcal{B}(\IR^k)$
\pagebreak[3]
\begin{align*}
	& \IP(\widehat{w}^n + Y_n \in B) \\
  &= Z_n^{-1} \sum_{\sigma \in \{\pm 1\}^N} \mu_{J_n}(\sigma) \int_{B} \exp \left( - \frac{1}{2} \skal{x-\widehat{w}^n, \Lambda_n(x-\widehat{w}^n)} \right) dx \\
	&= \frac{2^N}{C_n Z_n} \int_B \exp\left( - \frac{1}{2} \skal{x, \Lambda_n x} \right) \IE_{\mu_0} \exp\left( N \skal*{ \frac{1}{\sqrt{N}} \Gamma_n V^T \Lambda_n x, \frac{1}{N} \Gamma_n^{-2} m} \right) dx \\
	&= \frac{2^N}{C_nZ_n} \int_B \exp\left( - \Phi_n(x) \right) dx,
	\end{align*}
	where we have defined
	\begin{align*}
  \begin{split}
		\Phi_n(x) &\coloneqq \frac{1}{2} \skal{x, \Lambda_n x} - \sum_{i=1}^k \abs{B_i^{(n)}} \log \cosh \left( \frac{\sqrt{N}}{\abs{B_i^{(n)}}} (\Gamma_n V_n \Lambda_n x)_i \right) \\
		&= \frac{1}{2} \skal{x, \Lambda_n x} - \sum_{i=1}^k \abs{B_i^{(n)}} \log \cosh \left( \abs{B^{(n)}_i}^{-1/2} (V_n \Lambda_n x)_i \right).
    \end{split}
	\end{align*}
	Since $\log \cosh(x) = \frac{1}{2} x^2 + O(x^4)$, we obtain
	\begin{align} \label{eqn:ApproximationHamiltonian}
		\Phi_n(x) &= \frac{1}{2} \skal{x, \Lambda_n x} - \frac{1}{2} \skal{x, \Lambda_n^2 x} + \frac{1}{N} O\left(\sum_{i = 1}^k \frac{N}{\abs{B^{(n)}_i}} (V_n \Lambda_n x)_i^4 \right) \\
		&= \frac{1}{2} \skal{x, (\Lambda_n - \Lambda_n^2) x} + \frac{1}{N} O(\norm{\Gamma_n^{-1/2} V_n \Lambda_n x}_4^4). \notag
	\end{align}
	For parameters $r, R > 0$ let $B_{0,r,R} \coloneqq \{ x \in \IR^k : r \le \norm{x}_2^2 \le R \}$ and decompose
	\begin{align*}
		\IP(\widehat{w}^n + Y_n \in B) &= \frac{2^N}{C_n Z_n} \left( \int_{B \cap B_R(0)} + \int_{B \cap B_{0,R,r\sqrt{N}}} + \int_{B \cap B_{r\sqrt{N}}(0)^c} \right) \exp\left( - \Phi_n(x) \right) dx \\
		&\eqqcolon \frac{2^N}{C_n Z_n} \left( I_1 + I_2 + I_3 \right).
	\end{align*}
	Since $\Lambda_n \to \Lambda_\infty$ (which is a consequence of the continuity of the eigenvalues) we have for any $R > 0$
	\[
		\lim_{n \to \infty} I_1 = \int_{B \cap B_R(0)} \exp \left( - \frac{1}{2} \skal{x, (\Lambda_\infty - \Lambda_\infty^2) x} \right) dx.
	\]
	Next, we will estimate \eqref{eqn:ApproximationHamiltonian} from below in order to obtain an upper bound for $I_2$. If we define $C_{2,4} \coloneqq \norm{\Id}_{2 \to 4}$, it follows that
	\begin{align*}
  \begin{split}
		\Phi_n(x) &\ge \frac{1}{2} \skal{x, (\Lambda_n - \Lambda_n^2)x} - C(r) \norm{\Gamma_n^{-1/2}}_{4 \to 4} r^2 \norm{\Lambda_n}_{2 \to 2}^2 \skal{x,x}  \\
		&\ge \frac{1}{2} \skal{x, \left( \Lambda_n - \Lambda_n^2 - C(r)r^2 C \right) x} \\
		&\ge c \frac{1}{2} \skal{x, x}.
    \end{split}
	\end{align*}
	Here, we have used the convergence of $\Gamma_n$ to $\Gamma_\infty$ to bound $\norm{\Gamma_n^{-1/2}}_{4 \to 4}$ and the fact that $C(r)r^2 \to 0$ as $r \to 0$, so that the right hand side is positive definite for $r$ small enough, uniformly in $n$. Thus, after taking the limit $n \to \infty$, $I_2$ will vanish in the limit $R \to \infty$. \par
	Lastly, we need to show that $I_3$ vanishes as well. To this end, we show that we can choose $r > 0$ small enough to ensure that $\Phi_n(x) \ge \exp(-N c)$ uniformly for $x \in B_{r\sqrt{N}}(0)^c$ and for $n$ large enough. Since $\norm{\Lambda_n - \Lambda_\infty}_{2 \to 2} \to 0$ and $\norm{\Lambda_\infty}_{2 \to 2} < 1$, choose $n$ large enough so that $\norm{\Lambda_n}_{2 \to 2} < 1$ uniformly.
	Again, as before, it can be seen that $0$ is the only minimum for $n$ chosen that way. Indeed, after some manipulations any critical point satisfies $\Gamma_n A \Gamma_n \tanh(y) = y$, and since $\norm{\tanh(y)}_2 \le \norm{y}_2$ and $\norm{\Gamma_n A \Gamma_n}_{2 \to 2} < 1$, this is only possible for $y = 0$. As a consequence, for any $r > 0$ there is a constant $c$ such that uniformly $\tilde{\Phi}_n(x) \ge c$, i.e.
	\begin{equation*}	
		I_3 \le \int \eins_{\{\norm{x}_2 > r \sqrt{N}\}} \exp\left( - \Phi_n(x) \right) dx \le \int_{B_{r \sqrt{N}}(0)^c} \exp\left( -N \tilde{\Phi}_n(N^{-1/2}x) \right) dx \to 0.
	\end{equation*}
	Lastly, choose $r > 0$ so small that $\Lambda_n - \Lambda_n^2 - C(r)r^2 C$ is uniformly positive definite, and observe that we obtain
	\[
		\lim_{n \to \infty} \IP(\widehat{w}^n + Y_n \in B) = \mathcal{N}(0, (\Lambda_\infty - \Lambda_\infty^2)^{-1})(B).
	\]
	From here, it remains to undo the convolution (e.g. by using the characteristic function), giving
	\[
		\lim_{n \to \infty} \mu_{J_n}(\widehat{w}^n \in B) = \mathcal{N}(0, (\Id - \Lambda_\infty)^{-1})(B).
	\]
	With the help of Slutsky's theorem and the definition $\widehat{m}^n = V^T_n \widehat{w}^n$ this implies
	\[
		\mu_{J_n} \circ \widehat{m}^n \Rightarrow \mathcal{N}(0, V^T (\Id - \Lambda_\infty)^{-1} V) = \mathcal{N}(0, \left( \Id - \Gamma_\infty A \Gamma_\infty \right)^{-1})
	\]
	as claimed.
\end{proof}
	
\begin{exa*}
Consider the case $k = 2$ and
\[
	A_2 = \begin{pmatrix}
	\beta & \alpha \\ \alpha & \beta
	\end{pmatrix}.
\]
$A_2$ is positive definite if $\beta \ge 0$ and $(\beta - \alpha)(\beta + \alpha) \ge 0$, i.e.~if $\abs{\alpha} \le \beta$. We have the diagonalization
\[
	A_2 = \frac{1}{2} \begin{pmatrix}
	1 & 1 \\ 1 & -1
	\end{pmatrix}
	\begin{pmatrix}
	\beta+\alpha & 0 \\ 0 & \beta-\alpha
	\end{pmatrix}
	\begin{pmatrix}
	1 & 1 \\ 1 & -1
	\end{pmatrix} \eqqcolon V^T \Lambda V,
\]
and $w = V^T m = \frac{1}{\sqrt{2}} \begin{pmatrix} 1 & 1 \\ 1 & -1 \end{pmatrix}m$ corresponds to the transformation performed in \cite[Theorem 1.2]{LS18} (up to a factor of $\sqrt{2}$). In this case
	\[
		\left(\Id - \frac{1}{2}A_2\right)^{-1} = \frac{2}{(\beta-2)^2 - \alpha^2}\begin{pmatrix}
			2 - \beta & \alpha \\ \alpha & 2-\beta
		\end{pmatrix}
	\]
which is exactly the covariance matrix in \cite{LS18} (again up to a factor of $2$). Note that similar results have been derived in \cite{KT18}.
\end{exa*}

\begin{rem*}
	If $A \in M_k(\IR)$ is symmetric and positive semidefinite, then a variant of the proof shows that if we let $A = V^T \Lambda V$ with $\Lambda = \diag(\lambda_1, \ldots, \lambda_l, 0, \ldots, 0)$ for $l < k$, $((V\tilde{m})_i)_{i \le l}$ converges to an $l$-dimensional normal distribution with covariance matrix $\Sigma_l \coloneqq (\Id - \Lambda_l)^{-1}, \Lambda_l = \diag(\lambda_1, \ldots, \lambda_l)$. This can be applied to the matrix $A_2$ above with $\alpha = \beta$, resulting in a CLT for the magnetization in a Curie--Weiss model, which of course can also be obtained by choosing $k = 1$ and $0 < \beta < 1$.
\end{rem*}

\subsection{Non-central limit theorem}
Recall the situation of Theorem \ref{theorem:CLTcriticalcase}: The block interaction matrix has eigenvalues $0 < \lambda_1 \le \ldots \le \lambda_{k-1} < \lambda_k = k$ and we consider the uniform case, i.e.~$\Gamma_\infty^2 = k^{-1}$. Moreover, we use the definitions 
\begin{align*}
w' &= \diag(N^{-1/2}, \ldots, N^{-1/2}, N^{-3/4}) V m^{(n)}, \\
\hat{C}_N &= \diag(\lambda_1, \ldots, \lambda_{k-1}, kN^{1/2}),
\end{align*}
so that
\begin{equation*}
	H_n = \frac{1}{2} \skal{\hat{C}_N w', w'}.
\end{equation*}

\begin{proof}[Proof of Theorem \ref{theorem:CLTcriticalcase}]
Let $Y_n \sim \mathcal{N}(0, \hat{C}_N^{-1})$ and $X_n \sim \mu_{J_n}$ be independent random variables, defined on a common probability space. We have for any Borel set $B \in \mathcal{B}(\IR^k)$
\begin{align*}
	\IP\left( w_n'(X_n) + Y_n \in B \right) &= 2^{N} Z_n^{-1} \int_B \exp\left( - \frac{1}{2} \skal{\hat{C}_N x,x} \right) \IE_{\mu_0} \exp\left( \skal{x,\hat{C}w'} \right) dx \\
	&= \tilde{Z}_n^{-1} \int_B \exp\left( - \frac{1}{2} \skal{\hat{C}_N x,x} + \frac{N}{k} \sum_{i= 1}^k \log \cosh((V^T \Lambda \tilde{x})_i) \right) dx \\
	&= \tilde{Z}_n^{-1} \int_B \exp \left( - \Phi_N(x) \right) dx \\
	&= \tilde{Z}_n^{-1} \int_B \exp \left( - N \tilde{\Phi}_N\left( \frac{x_1}{N^{1/2}}, \ldots, \frac{x_{k-1}}{N^{1/2}}, \frac{x_k}{N^{1/4}} \right) \right) dx
\end{align*}
where we used
\begin{align*}
	\Phi_N(x) &\coloneqq \frac{1}{2} \skal{x,\hat{C}_N x} - \frac{N}{k} \sum_{i = 1}^k \log \cosh\left( \left(V^T \Lambda \left(\frac{x_1}{N^{1/2}}, \ldots, \frac{x_{k-1}}{N^{1/2}}, \frac{x_k}{N^{1/4}}\right)\right)_i \right), \\
	\tilde{\Phi}_N(x) &\coloneqq \frac{1}{2} \skal{x, \Lambda x} - \frac{1}{k} \sum_{i = 1}^k \log \cosh \left( (V^T \Lambda x)_i \right).
\end{align*}
Now the proof is along the same lines as the proof of the CLT in the high temperature phase, with the slight modification that we use expansion of $\log \cosh$ to fourth order
\[
	\log \cosh(x) = \frac{x^{2}}{2} - \frac{x^4}{12} + O(x^6).
\]
We again split $\IR^k$ into three regions, namely the inner region $I_1 = B_R(0)$ for an arbitrary $R > 0$, the intermediate region $I_2 = K_r \backslash B_R(0)$ for some arbitrary $r > 0$, where
\[
K_r \coloneqq \left\lbrace x \in \IR^k : \norm*{\left( N^{-1/2}x_1, \ldots, N^{-1/2}x_{k-1}, N^{-1/4}x_k\right)}_\infty \le r \right\rbrace,
\]
and the outer region $I_3 \coloneqq K_r^c$. Also define the rescaled vector
\begin{align*}
  \tilde{x} \coloneqq \left( \lambda_1 N^{-1/2} x_1, \ldots, \lambda_{k-1} N^{-1/2} x_{k-1}, k N^{-1/4} x_k \right).
\end{align*}
Firstly, in the inner region we rewrite
\begin{align*}
\Phi_N(x) &= \frac{1}{2} \skal{x,\hat{C}_N x} - \frac{N}{2k} \sum_{i = 1}^k (V^T \tilde{x})_i^2 + \frac{N}{12k} \sum_{i = 1}^k (V^T \tilde{x})_i^4 + \frac{N}{k} O(\norm{V^T \tilde{x}}_6^6) \\
&= \frac{1}{2} \sum_{i = 1}^{k-1} \left(\lambda_i - \frac{\lambda_i^2}{k}\right) x_i^2 +  \frac{N}{12k} \norm{V^T \tilde{x}}_4^4 + \frac{N}{k} O(\norm{V^T \tilde{x}}_6^6) \\
&= \frac{1}{2} \sum_{i = 1}^{k-1} \left(\lambda_i - \frac{\lambda_i^2}{k}\right) x_i^2 + \frac{k^3}{12} x_k^4 \sum_{i=1}^k V_{ki}^4 + O(N^{-1/4}) + \frac{N}{k} O(\norm{V^T \tilde{x}}_6^6),
\end{align*}
and since the convergence of the error terms is uniform on any compact subset of $\IR^k$, for any fixed $R > 0$ this yields
\begin{align*}
	\lim_{N \to \infty} \int_{B \cap I_1} \exp \left( - \Phi_N(x) \right) dx = \int_{B \cap I_1} \exp \left(- \frac{1}{2} \sum_{i = 1}^{k-1} \left(\lambda_i - \frac{\lambda_i^2}{k}\right) x_i^2 - \frac{k^3}{12} x_k^4 \sum_{i=1}^k V_{ki}^4 \right) dx.
\end{align*}
Secondly, we show that the outer region does not contribute to the limit $N \to \infty$. It can be seen by elementary tools that $\tilde{\Phi}_N$ has a unique minimum $0$ in $0$, and so for any $r > 0$ we have $\inf_{x \in I_3} \tilde{\Phi}(x) > 0$. Using the monotone convergence theorem, we obtain
\[
\lim_{N \to \infty} \int_{B \cap I_3} \exp \left( - N \tilde{\Phi}(x) \right) dx = 0.
\]
Lastly, we will estimate the contribution of the intermediate region from above by a quantity which vanishes as $R \to \infty$. To this end, we will bound the function $\Phi_N$ from below. Recall that
\begin{align*}
	\Phi_N(x) &= \frac{1}{2} \skal{x,\hat{C}_N x} - \frac{N}{2k} \sum_{i = 1}^k (V^T \tilde{x})_i^2 + \frac{N}{12k} \sum_{i = 1}^k (V^T \tilde{x}_i)^4 + \frac{N}{k} O(\norm{V^T \tilde{x}}_6^6) \\
	&= \frac{1}{2} \skal{x, \hat{C}_N x} - \frac{N}{2k} \skal{\tilde{x}, \tilde{x}} + \frac{N}{12k} \norm{V^T \tilde{x}}_4^4 + \frac{N}{k} O(\norm{V^T \tilde{x}}_6^6)
\end{align*}
and since $\norm{V^T \tilde{x_i}}_4^4 \ge C \norm{\tilde{x}_i}^4_4$ for $C = \norm{V}_{4 \to 4}^{-4}$ this yields
\begin{align*}
	\Phi_N(x) &\ge \frac{1}{2} \skal{x, \hat{C}_N x} - \frac{N}{2k} \skal{\tilde{x}, \tilde{x}} + \frac{N}{12k} C \norm{\tilde{x}}_4^4 + \frac{N}{k} O(\norm{V^T \tilde{x}}_6^6) \\
	&= \frac{1}{2} \skal{\left( \Lambda - k^{-1} \Lambda \right)x,x} + \frac{k^4}{12} C x_k^4 + O(\norm{V^T \tilde{x}}_6^6).
\end{align*}
Now, as in the case of the central limit theorem, we can estimate from below the error term in such a way that there is a positive constant $c$ and a positive definite matrix $C$ such that
\begin{equation*}
	\Phi_N(x) \ge \frac{1}{2} \skal{C(x_1, \ldots, x_{k-1},0),(x_1, \ldots, x_{k-1},0)} + c x_k^4,
\end{equation*}
from which we obtain an upper bound, i.e.
\begin{equation*}
	\int_{B \cap I_3} \exp \left( - \Phi_N(x) \right) dx \le \int_{B \cap I_3} \exp \left( - \frac{1}{2} \sum_{i,j=1}^{k-1} C_{ij} x_i x_j - c x_k^4 \right) dx,
\end{equation*}
and the right hand side vanishes as $R \to \infty$ by dominated convergence. As a result, the limit $n \to \infty$ exists and is equal to
\[
	\lim_{n \to \infty} \IP\left( w'_n(X_n) + Y_n \in B \right) = Z^{-1} \int_B \exp\left( - \frac{1}{2} \sum_{i = 1}^{k-1} \left( \lambda_i - \frac{\lambda_i^2}{k} \right) x_i^2 - \frac{k^3}{12} x_k^4 \sum_{i = 1}^k V_{ki}^4 \right) dx.
\]
The convergence results for the non-convoluted vector follow easily by considering the characteristic functions. We have for any $t \in \IR^k$
\begin{equation*}
	\IE \exp\left( i \skal{t,w_n'(X_n) + Y_n} \right) \to \exp\left( - \frac{1}{2} \skal{(t_1,\ldots, t_{k-1}, \tilde{\Sigma} (t_1, \ldots, t_{k-1}))}\right) \phi(t_k),
\end{equation*}
where $\tilde{\Sigma} = \diag\left( \lambda_i^{-1} + (k-\lambda_i)^{-1} \right)$ and $\phi$ is the characteristic function of a random variable with distribution $\exp\left(-x_k^4 k^3/12 \sum_{i = 1}^k V_{ki}^4 \right)$. Using the independence of $X_n$ and $Y_n$, the results follow by simple calculations.
\end{proof}

\subsection{Central limit theorem: Stein's method}
Lastly, we will prove Theorem \ref{theorem:CLTviaSteinsMethod} using Stein's method of exchangeable pairs. For brevity's sake, for the rest of this section we fix $n \in \IN$ and we will drop all sub- and superscripts (e.g.~we write $B_i$ instead of $B_i^{(n)}$, $\hat{m}$ instead of $\hat{m}^{(n)}$, $J$ instead of $J_n$ et cetera).
It is more convenient to formulate this approach in terms of random variables. Let $X$ be a random vector with distribution $\mu_J$ and $I$ be an independent random variable uniformly distributed on $\{1,\ldots, N\}$. First, denote by $(X, \tilde{X})$ the exchangeable pair which is given by taking a step in the Glauber chain for $\mu_J$, i.e.~$\tilde{X}$ is the vector after replacing $X_I$ by an independent $\tilde{X}_I$ with distribution $\tilde{X}_I \sim \mu_J( \cdot \mid \overline{X}_I)$ (the exchangeability follows from the reversibility of the Glauber dynamics). Consequently, $(\hat{m}, \hat{m}') = (\hat{m}(X), \hat{m}(\tilde{X}))$ is also exchangeable. More precisely, with the standard basis vectors $(e_i)_{i = 1,\ldots,k}$ of $\IR^k$ we have
\begin{equation}		\label{eqn:defi:mtildeminusm}
\hat{m}' \coloneqq \hat{m} - \frac{X_I - \tilde{X}_I}{\sqrt{\abs{B_I}}} \begin{pmatrix}1 \\ \vdots \\ 1 \end{pmatrix} \Rightarrow \hat{m} - \hat{m}' = \frac{X_I - \tilde{X}_I}{\sqrt{M}} e_{h(I)}.
\end{equation}
We need the following lemma to identify the conditional expectation of $\tilde{X}_i$. Here, we write $h: \{1,\ldots, N\} \to \{1,\ldots,k\}$ for the function that assigns to each position its block, i.e.~$h(j) = k \Longleftrightarrow j \in B_k$.

\begin{lemma}	\label{lemma:condProb}
Let $\mathcal{F} = \sigma(X)$ and $(X, \tilde{X})$ be defined as above. Then for each fixed $i \in \{1, \ldots, N\}$
\[
	\IE \left( \tilde{X}_i \mid \mathcal{F} \right) = \tanh\left( \frac{1}{\sqrt{N}} (A\Gamma \hat{m})_i - \frac{1}{N} A_{h(i)h(i)} X_i \right).
\]
\end{lemma}

\begin{proof}
	For any Ising model $\mu = \mu_J$ the conditional distribution of $\tilde{X}_i$ is given by $\mu(\cdot \mid \overline{X}_i)$ and so
	\[
		\IE \left( \tilde{X}_i \mid \mathcal{F} \right) = 2\mu(1 \mid \overline{X}_i) - 1 = \tanh\left( (J^{(d)}X)_i \right),
	\]
	where we recall the notation $J^{(d)}$ for the matrix without its diagonal, i.e.~$J^{(d)} = J - \diag(J_{ii})$. In the case that $J = J_n$ is the block model matrix, this yields
	\begin{align*}
	\IE\left( \tilde{X}_i \mid \mathcal{F} \right)
	&= \tanh\left( N^{-1} \sum_{j = 1}^k A_{h(i)j} \sum_{l \in B_j} X_l - N^{-1} A_{h(i) h(i)} X_i \right) \\
  &= \tanh\left( N^{-1} (A m)_{h(i)} - N^{-1} A_{h(i) h(i)} X_i \right) \\
	&= \tanh\left( N^{-1/2} (A\Gamma \hat{m})_i - N^{-1} A_{h(i)h(i)} X_i \right).
	\end{align*}
\end{proof}

Since the conditional expectation will be of importance, we define
\begin{equation*}
	g_i(X) \coloneqq N^{-1} (A m)_{h(i)} - N^{-1} A_{h(i) h(i)} X_i = N^{-1/2} (A\Gamma \hat{m})_i - N^{-1} A_{h(i)h(i)} X_i,
\end{equation*}
so that $\IE(\tilde{X}_i \mid \mathcal{F}) = \tanh(g_i(X))$. Note that $g_i$ actually does not depend on $X_i$, the latter term is added for convenience to rewrite the first term. Thus we have $g_i(X) = \IE(\tilde{X}_i \mid \overline{X}_i)$.

\begin{lemma}\label{lemma:conditionalExpectation}
We have
\begin{equation*}
	\IE\left( \hat{m} - \hat{m}' \mid \mathcal{F}\right) = N^{-1} \left( \Id - \Gamma A \Gamma \right) \hat{m} + R(X),
\end{equation*}
with
\begin{equation*}
R(X) \coloneqq N^{-1} \sum_{i = 1}^k e_i \left( (\Gamma A \Gamma \hat{m})_i - \abs{B_i}^{-1/2} \sum_{j \in B_i} \tanh\left( g_j(X) \right) \right).
\end{equation*}
\end{lemma}

\begin{proof}
From equation \eqref{eqn:defi:mtildeminusm} and Lemma \ref{lemma:condProb} we obtain
\begin{align*}
	\IE\left( \hat{m} - \hat{m}' \mid \mathcal{F} \right) &= N^{-1} \sum_{i = 1}^k e_i \abs{B_i}^{-1/2}\sum_{j \in B_i} \IE( X_j - \tilde{X_j} \mid \mathcal{F}) \\
	&= N^{-1} \sum_{i = 1}^k e_i \hat{m}_i - N^{-1} \sum_{i = 1}^k e_i \abs{B_i}^{-1/2} \sum_{j \in B_i} \tanh(g_j(X)) \\
	&= N^{-1} \hat{m} - N^{-1} \sum_{i = 1}^k e_i \abs{B_i}^{-1/2} \Big( \sum_{j \in B_i} N^{-1/2} (A\Gamma \hat{m})_i \Big) + R(X) \\
	&= N^{-1} \left( \Id - \Gamma A \Gamma \right) \hat{m} + R(X).
\end{align*}
\end{proof}

For $n$ large enough, the matrix $\Lambda \coloneqq N^{-1}(\Id - \Gamma A \Gamma)$ satisfies $\norm{\Lambda}_{2 \to 2} < \frac{1}{N}$ and is thus invertible, with inverse $\Lambda^{-1} = N \sum_{l = 0}^\infty (\Gamma A \Gamma)^l$. Moreover, we also have $\norm{\Lambda^{-1}}_{2 \to 2} \le N (1-\norm{\Gamma A \Gamma}_{2 \to 2})^{-1}$.

We will need the following approximation theorem for random vectors.

\begin{thm}[\cite{RR09}, Theorem 2.1]
Assume that $(W,W')$ is an exchangeable pair of $\IR^d$-valued random vectors such that
\begin{equation*}
	\IE W= 0, \quad \quad \IE W W^t = \Sigma,
\end{equation*}
with $\Sigma \in \IR^{d\times d}$ symmetric and positive definite. Suppose further that 
\begin{equation*}
\IE[W' - W \mid W] = -\Lambda W + R
\end{equation*}
is satisfied for an invertible matrix $\Lambda$ and a $\sigma(W)$-measurable random vector $R$. Then, if $Z$ has $d$-dimensional standard normal distribution, we have for every three times differentiable function
\begin{equation*}
	\abs{\IE h(W) - \IE h(\Sigma^{1/2} Z)} \le \frac{\abs{h}_2}{4} E_1 + \frac{\abs{h}_3}{12} E_2 + \left( \abs{h}_1 + \frac{1}{2} d \norm{\Sigma}^{1/2} \abs{h}_2 \right) E_3,
\end{equation*}
where, with $\lambda{(i)} \coloneqq \sum_{m = 1}^d \abs{\left( \Lambda^{-1} \right)_{m,i}}$, we define the three error terms
\begin{align*}
	E_1 &= \sum_{i,j = 1}^d \lambda{(i)} \sqrt{\Var \IE\left[(W_i'-W_i)(W_j'-W_j) \mid W \right]}, \\
	E_2 &= \sum_{i,j,k = 1}^d \lambda{(i)} \IE \abs{(W_i'-W_i)(W_j'-W_j)(W_k'-W_k)}, \\
	E_3 &= \sum_{i = 1}^d \lambda{(i)} \sqrt{\Var R_i}.
\end{align*}
\end{thm}
Here, $\abs{h}_j$ denotes the supremum of the partial derivatives of up to order $j$.

Note that in the proof the choice of $\sigma(W)$ for the conditional expectation is arbitrary; it suffices to take any $\sigma$-algebra $\mathcal{F}$ with respect to which $W$ is measurable. Clearly, the value $E_1$ has to be adjusted accordingly. 

\begin{cor}			\label{corollary:differencemXandZFormulation}
Let $\hat{m}$ be the block magnetization vector and $\hat{m}'$ as above, define $\Sigma \coloneqq \IE \hat{m}\hat{m}^T$ and let $Z \sim \mathcal{N}(0, \Sigma)$. For any function $h \in \mathcal{F}_3$
\begin{equation*}		
	\abs{\IE h(\hat{m}(X)) - \IE h(Z)} \le CN \left( \frac{\abs{h}_2}{4} E_1 + \frac{\abs{h}_3}{12} E_2 + \left(\abs{h}_1 + \frac{1}{2} k \norm{\Sigma}^{1/2} \abs{h}_2 \right) E_3 \right)
\end{equation*}
with the three error terms
\begin{align*}
	E_1 &= \sum_{i = 1}^k \sqrt{\Var\left( \IE((\hat{m}_i(X) - \hat{m}_i(\tilde{X}))^2 \mid \mathcal{F}) \right)} \\
	E_2 &= \sum_{i = 1}^k \IE \abs{\hat{m}_i(X) - \hat{m}_i(\tilde{X})}^3 \\
	E_3 &= \sum_{i = 1}^k \sqrt{\Var(R_i)}.
\end{align*}
\end{cor}

Finally, the following lemma shows that all error terms $E_i$ can be bounded by a term of order $N^{-3/2}$.

\begin{lemma}			\label{lemma:BoundOnEi}
	In the situation of Corollary \ref{corollary:differencemXandZFormulation} we have
	\begin{equation*}
		\max(E_1,E_2,E_3) = O(N^{-3/2}).
	\end{equation*}
\end{lemma}

Before we prove this lemma (and consequently Theorem \ref{theorem:CLTviaSteinsMethod}), we will state concentration of measure results in the block spin Ising models. These will be necessary to bound $E_1, E_2, E_3$. The first step is the existence of a logarithmic Sobolev inequality for the Ising model $\mu_{J_n}$ with a constant that is uniform in $n$.

\begin{proposition}			\label{proposition:LSIforIsingModels}
	Under the general assumptions, if $\norm{\Gamma_\infty A \Gamma_\infty}_{2 \to 2} < 1$, then for $n$ large enough the Ising model $\mu_{J_n}$ satisfies a logarithmic Sobolev inequality with a constant $\sigma^2 = \sigma^2(\norm{\Gamma_\infty A \Gamma_\infty}_{2 \to 2})$, i.e.~for any function $f: \{-1,+1\}^N \to \IR$ we have
	\begin{equation}\label{eqn:LSIforIsing}
		\Ent_{\mu_{J_n}}(f^2) \le 2\sigma^2 \sum_{i = 1}^N \IE_{\mu_{J_n}} (f - f \circ T_i)^2,
	\end{equation}
	where $\Ent$ is the entropy functional and $T_i: \{-1,+1\}^N \to \{-1,+1\}^N, (\sigma_1, \ldots, \sigma_N) \mapsto (\sigma_1,\ldots, \sigma_{i-1}, -\sigma_i, \sigma_{i+1}, \ldots, \sigma_N)$ the sign flip operator.
\end{proposition}

This follows immediately from \cite[Proposition 1.1]{GSS18}, since $\Gamma_n A \Gamma_n \to \Gamma_\infty A \Gamma_\infty$, which implies the convergence of the norms, i.e.~for $n$ large enough we have $\norm{\Gamma_n A \Gamma_n}_{2 \to 2} < 1$. Although the condition in \cite{GSS18} is $\norm{J}_{1 \to 1} < 1$, this was merely for applications' sake and $\norm{J}_{2 \to 2} < 1$ is sufficient to establish the logarithmic Sobolev inequality.

For any function $f: \{-1,+1\}^N \to \IR$ and any $r \in \{1,\ldots,N\}$ we write
\[
	\mathfrak{h}_r f(x) = \abs{f(x) - f(T_r x)},
\]
so that \eqref{eqn:LSIforIsing} becomes
\[
\Ent_{\mu_{J_n}}(f^2) \le 2\sigma^2 \sum_{r = 1}^N \int (\mathfrak{h}_r f(x))^2 d\mu_{J_n}(x).
\]
Moreover, it is known that \eqref{eqn:LSIforIsing} implies a Poincar{\'e} inequality
\begin{equation}\label{eqn:PoincareforIsing}
	\Var(f) \le \sigma^2 \sum_{r = 1}^N \IE \mathfrak{h}_r f(X)^2.
\end{equation}

\begin{proof}[Proof of Lemma \ref{lemma:BoundOnEi}]
	\textbf{Error term $\mathbf{E_1}$:} To treat the term $E_1$, fix $i \in \{1,\ldots, k\}$ and observe that
	\begin{align*}
		\IE \left( (\hat{m}_i(X) - \hat{m}_i(\tilde{X}))^2 \mid \mathcal{F} \right) &= N^{-1} \sum_{j = 1}^N \IE \left( (\hat{m}_i(X) - \hat{m}_i(\overline{X}_j, \tilde{X}_j))^2 \mid \mathcal{F} \right) \\
		&= (N\abs{B_i})^{-1} \sum_{j \in B_i} \IE \left( (X_j - \tilde{X}_j)^2 \mid \mathcal{F} \right) \\
		&= - 2(N\abs{B_i})^{-1} \sum_{j \in B_i} X_j \tanh(g_j(X)) + 2N^{-1}.
	\end{align*}
	Thus, if we define
	\begin{equation*}
	f_i(X) \coloneqq \abs{B_i^{(n)}}^{-1/2} \sum_{j \in B_i} X_j \tanh \left( N^{-1} \sum_{l = 1}^k A_{il} m_l(X) - N^{-1} A_{ii} X_i \right),
	\end{equation*}
	we see that
		\[
		\\Var^{1/2}\left( \IE \left( (\hat{m}_i^{(n)} - \hat{m}_i^{(n)}\, ')^2 \mid \mathcal{F} \right) \right) = 2N^{-1} \abs{B_i^{(n)}}^{-1/2} \Var^{1/2}(f_i(X)),
		\]
	and we need to show that $\Var(f_i(X)) = O(1)$. Using the Poincar{\'e} inequality \eqref{eqn:PoincareforIsing} it suffices to prove that $\mathfrak{h}_r f_i(X)^2 \le C \abs{B_i^{(n)}}^{-1}$. \par
	Let $r \in \{1,\ldots, N\}$ be arbitrary and define $h_i(X) \coloneqq N^{-1} \sum_{l = 1}^k A_{il} m_l(X) - N^{-1} A_{ii} X_i$. The first case is that $r \in B_i^{(n)}$, for which
	\begin{align*}
	  \mathfrak{h}_r f_i(X) &\le \abs{B_i}^{-1/2} \abs{2X_r \tanh(h_i(X))} + \abs{B_i}^{-1/2} \sum_{\substack{j \in B_i \\ j \neq r}} \abs{\tanh(h_i(X)) - \tanh(h_i(T_r X))} \\
	  &\le 4 \abs{B_i}^{-1/2} + \abs{B_i}^{-1/2} N^{-1} \sum_{\substack{j \in B_i\\j \neq r}} \abs*{\sum_{l = 1}^k A_{il} (m_l(X) - m_l(T_r(X)))} \\
	  &\le \abs{B_i}^{-1/2} (4 + 2\norm{A}_\infty).
	\end{align*}
	The second case $r \notin B_i^{(n)}$ follows by similar reasoning. \par
	\parindent0mm\textbf{Error term $\mathbf{E_2}$:} The second term $E_2$ is much easier to estimate, as
	\[
	\IE \abs{\hat{m}_i - \hat{m}_i'}^3 = N^{-1} \abs{B_i}^{-3/2} \sum_{j \in B_i} \IE \abs{X_j - \tilde{X}_j}^3 \le 8N^{-1} \abs{B_i}^{-1/2} = O(N^{-3/2}).
	\]
	\textbf{Error term $\mathbf{E_3}$:} To estimate the variance of the remainder term $R$ we first split it into two sums. For any $i = 1, \ldots, k$ write
	\begin{align*}
		R_i(X) &= N^{-1} \Big( \abs{B_i}^{-1/2} \sum_{j \in B_i} g_j(X) - \tanh(g_j(X)) + N^{-1} A_{ii} X_j \Big) \\
		&= N^{-1} \abs{B_i}^{-1/2} \sum_{j \in B_i} g_j(X) - \tanh(g_j(X)) + N^{-2} A_{ii} m_i(X) \\
		&\eqqcolon R_j^{(1)}(X) + R_j^{(2)}(X).
	\end{align*}
	Clearly $\norm{R_i - \IE R_i}_2 \le \norm{R_i^{(1)} - \IE R_i^{(1)}}_2 + \norm{R_i^{(2)} - \IE R_i^{(2)}}_2$ and we estimate these terms separately. It is obvious that the $L^2$ norm of the second term is of order $O(N^{-2})$. To estimate $R^{(1)}_i$, we use $\tanh(x) - x = O(x^3)$ to obtain
	\begin{align*}
		\norm{R^{(1)}_i - \IE R^{(1)}_i}_2 &\le CN^{-1} \abs{B_i}^{-1/2} \sum_{j \in B_i} \norm{ \abs{g_j(X)}^3}_2 \\
		&\le C N^{-1} \abs{B_i}^{-1/2} \sum_{j \in B_i} \norm{\abs{N^{-1/2} (A\Gamma \hat{m})_j}^3}_2 + \norm{N^{-3} \abs{A_{ii}}^3}_2 \\
		&= O(N^{-2}) + O(N^{-5/2}).
	\end{align*}
	In the last line we have used the fact that $\norm{(A \Gamma \hat{m})_i^3}_2 = \norm{(A \Gamma \hat{m}_i}_6^3$ and for all $p \ge 2$
	\[
		\norm{(A \Gamma \hat{m})_i}_p \le C \sum_{l = 1}^k \norm{\hat{m}_l}_p \le C \sum_{l = 1}^k (\sigma^2 p)^{1/2}
	\]
	which evaluated at $p = 6$ gives $\norm{(A\Gamma \hat{m})_i}_6^3 = O(1)$. For the details see \cite{GSS18}. The constant depends on a norm of $A \Gamma$, which by convergence to $A \Gamma_\infty$ can again be chosen independently of $n$.
\end{proof}

\begin{proof}[Proof of Theorem \ref{theorem:CLTviaSteinsMethod}]
The theorem follows immediately from Corollary \ref{corollary:differencemXandZFormulation} and Lemma \ref{lemma:BoundOnEi}.
\end{proof}

\section{Discussion and open questions}		\label{section:Discussion}
Although the questions raised in the introduction have been answered to a certain degree, there are still open questions that we were not yet able to answer. \par
The first question concerns the maxima of the rate function $I$. Firstly, note that by \cite[Theorem A.1]{CET05} the global maxima of $I$ are related to the global minima of the so-called \emph{pressure functional}, which can for example be found in \cite[equation (14)]{FC11}. Using the compactness of $[-1,1]^k$ and the continuity of $I$, the existence of a maximiser easily follows, but the number of maximisers is still obscure. From real-analyticity of $I$, we can infer that the set of maximisers is a $\lambda^k$ null set, but it could in principle contain infinitely many points. However, Lemmas \ref{lemma:2BlocksNumberOfMaxima} and \ref{lemma:BlockInteractionTwomaximisers} as well as numerics suggest that for positive interactions and $k \ge 2$, the number of local minima is twice the number of independent systems - see Figures \ref{figure:LDPk3} for the $k = 3$ and \ref{figure:LDPk2} for the $k = 2$ case below. 

However, we believe that the case of negative interactions between groups might drastically change the picture. Indeed, consider a model with three blocks and positive interaction $\beta$ within the blocks and negative interaction $\alpha$ between the blocks. Then, if $\beta$ is large enough, the points within the blocks will tend to be aligned. However, as $\alpha$ is negative, the magnetizations of block one and two will try to have different signs, but so do the magnetizations of blocks two and three, and three and one. Hence, frustration occurs. In this respect, a model with positive and negative interactions carries features of a spin glass.   

\begin{figure}[!ht]
	\centering
  \begin{subfigure}[c]{0.49\textwidth}
  \includegraphics[width=\textwidth]{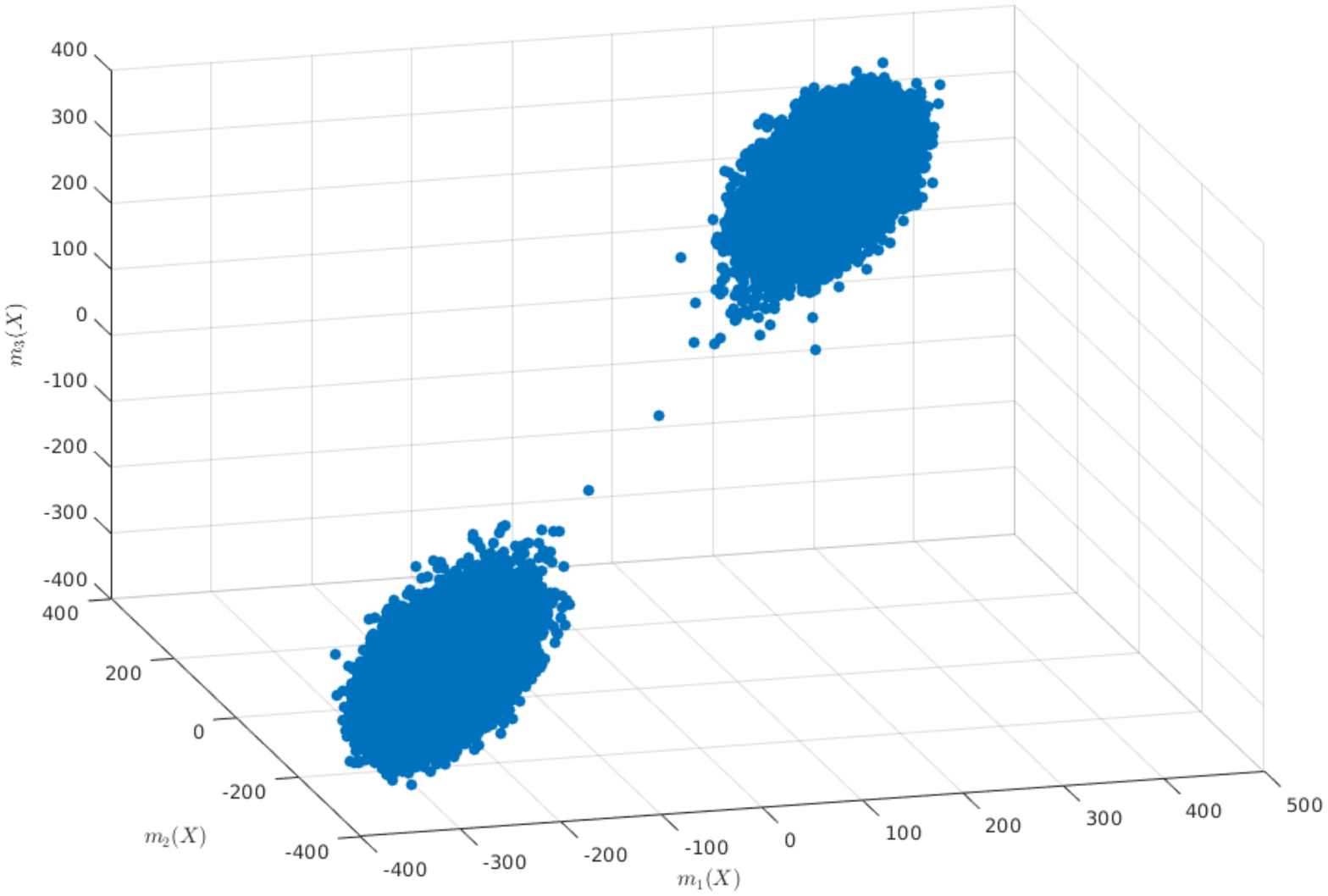}
  \caption{Block interaction $A = \begin{pmatrix} 2 & 0.9 & 0.7 \\ 0.9 & 2 & 0.4 \\ 0.7 & 0.4 & 2 \end{pmatrix}$}
	\end{subfigure}
	\begin{subfigure}[c]{0.49\textwidth}
  \includegraphics[width=\textwidth]{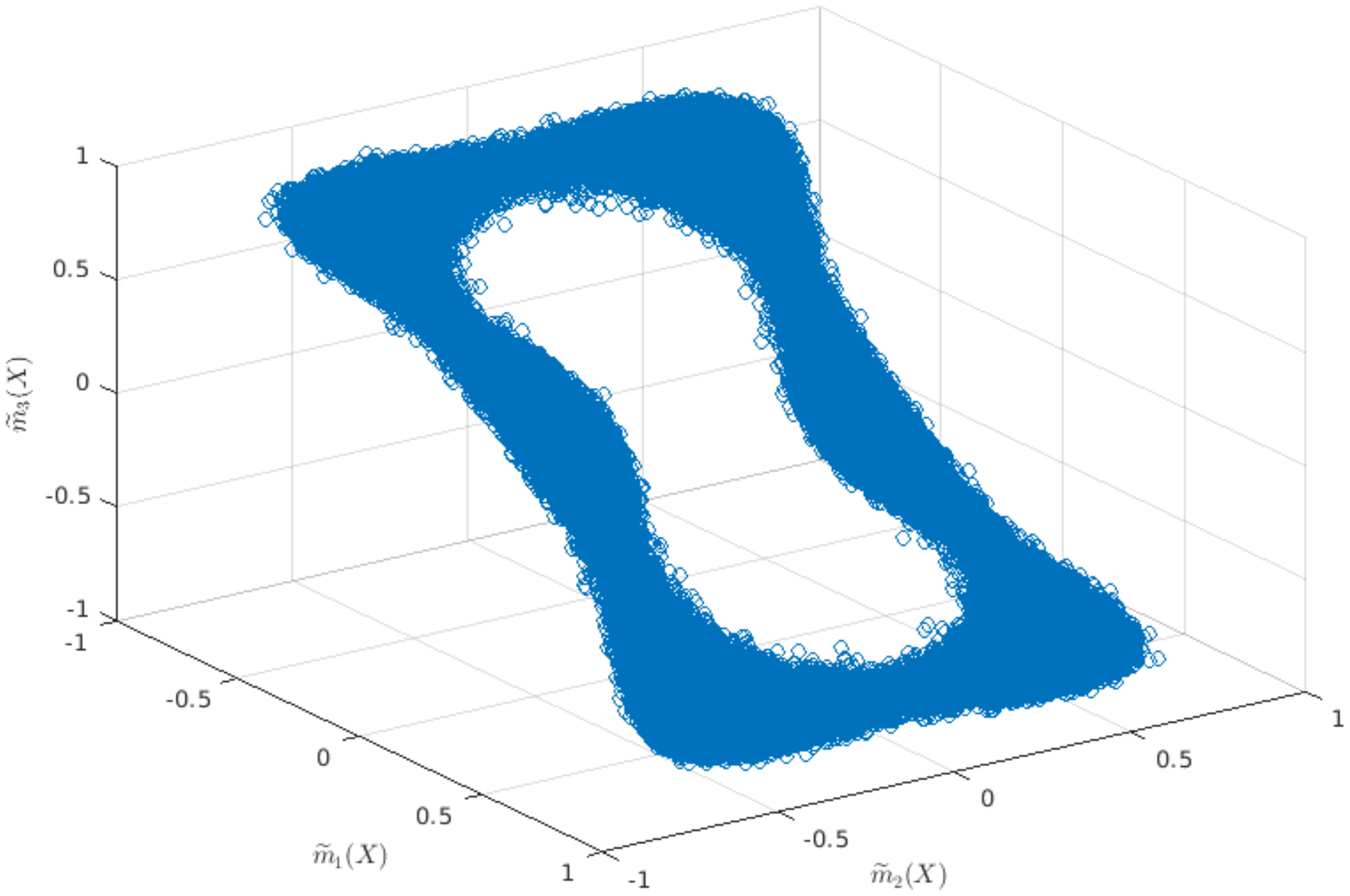}
  \caption{Block interaction $A = \begin{pmatrix} 3 & -1 & -1 \\ -1 & 3 & -1 \\ -1 & -1 & 3 \end{pmatrix}$}
	\end{subfigure}
	\caption[A scatterplot of the block magnetization in the uniform case]{A scatterplot of the normalized block magnetization $\tilde{m}$ in the uniform case with $k = 3$ blocks and $n = 500$, sampled using the Glauber dynamics -- note that it is not rapidly mixing!}
	\label{figure:LDPk3}
\end{figure}
\begin{figure}[!ht]
\centering
	\begin{subfigure}[c]{0.4\textwidth}
		\includegraphics[width=\textwidth]{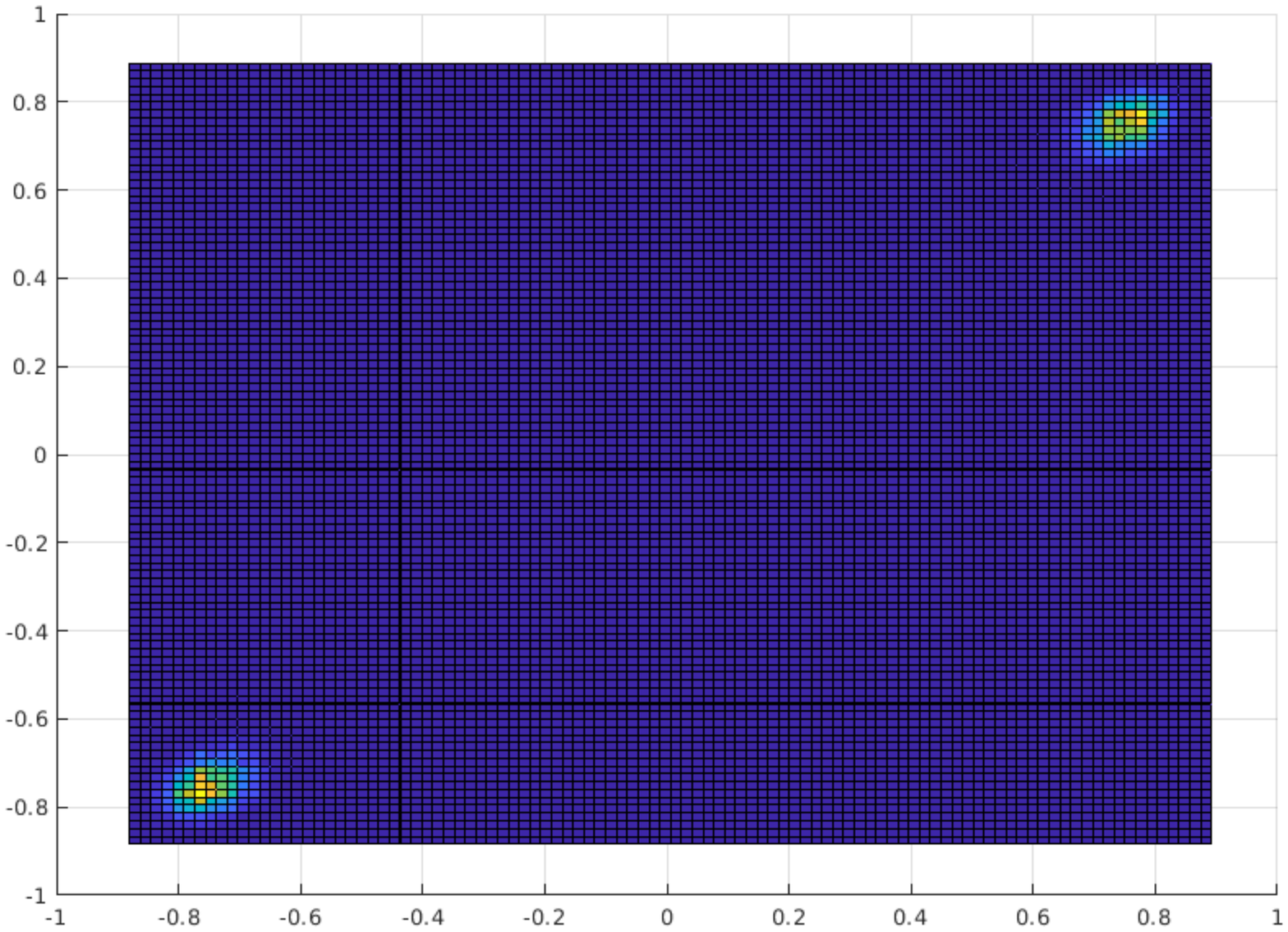}	
	\end{subfigure}
	\begin{subfigure}[c]{0.4\textwidth}
		\includegraphics[width=\textwidth]{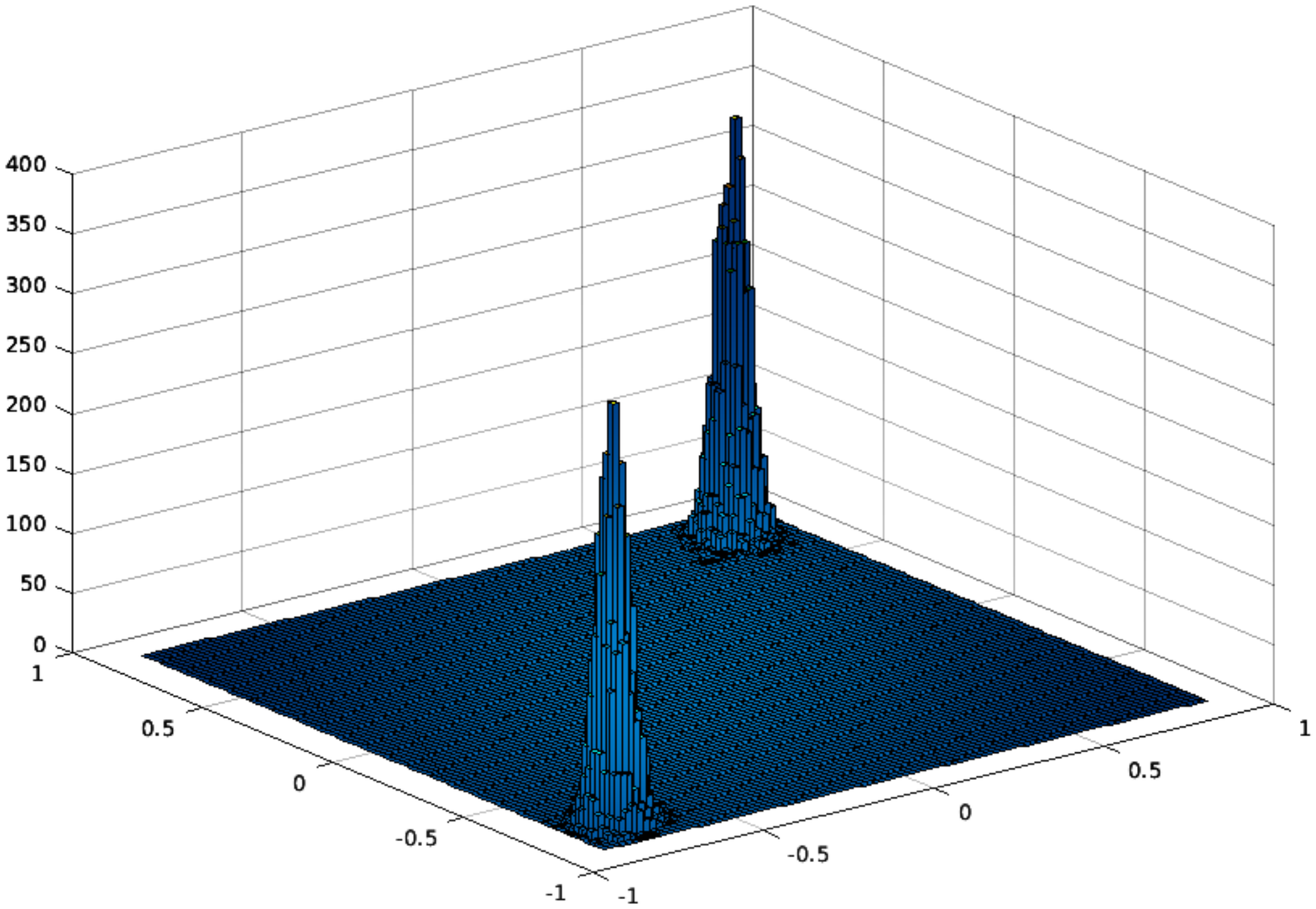}
	\end{subfigure}
	\caption[A heatmap (left) and a histogram (right) of the block magnetization vector in the uniform, low temperature case.]{A heatmap (left) and a histogram (right) of the block magnetization vector $m = (m^1, m^2)$ in the uniform, low temperature case. The block interaction matrix is given by $A = \begin{pmatrix}
	 	 	 1.8 & 0.8 \\
	 	  	0.8 & 1.8
	 \end{pmatrix}$.}
	 \label{figure:LDPk2}
\end{figure}

Another question is the relationship of Theorems \ref{theorem:CLTgeneralCase} and \ref{theorem:CLTviaSteinsMethod}. In Theorem \ref{theorem:CLTviaSteinsMethod} we consider the distance to a normal distribution with covariance matrix $\Sigma_n \coloneqq \IE \hat{m}^{(n)} (\hat{m}^{(n)})^T$ and not to $\Sigma_{\infty} \coloneqq (\Id - \Gamma_\infty A \Gamma_\infty)^{-1}$, which is the covariance matrix of the limiting distribution. Testing against functions $h \in \mathcal{C}_c^\infty(\IR^k)$, we see that $\Sigma_\infty$ is the limit of the matrices $\Sigma_n$. It is an interesting task to provide suitable bounds of $\norm{\Sigma_n - \Sigma_{\infty}}$ in any matrix norm, since \cite[Proposition 2.8]{RR09} provides bounds of $\abs{\IE h(X) - \IE h(Y)}$ for two random vectors with $X \sim \mathcal{N}(0,\Sigma_0)$ and $Y \sim \mathcal{N}(0, \Sigma_1)$ in terms of the $1$-distance of $\Sigma_0$ and $\Sigma_1$. \par
Thirdly, it remains an open problem to quantify the distance to a normal distribution with the ``limiting'' covariance matrix $\Sigma_\infty$. The central limit theorem in the one-dimensional Curie--Weiss model has been solved for example in \cite[Corollary 2.9]{EL10}. Therein one can see that the limiting covariance is $(1-\beta)^{-1}$ by considering the approximate linear regression condition. A similar condition is true in the multidimensional case. For example, in Lemma \ref{lemma:conditionalExpectation} we have proven
\begin{equation}
\IE \left( \hat{m}^{(n)} - \hat{m}^{(n)}\, ' \mid \mathcal{F} \right) = \lambda \Lambda^{-1} \hat{m}^{(n)} + R(X),
\end{equation}
where $\lambda = N^{-1}$ and $\Lambda = (\Id - \Gamma_n A \Gamma_n)^{-1}$. Thus, in the case $\Gamma_n \equiv \Gamma_\infty$ (e.g.~consider a subsequence along which this holds) $\Lambda$ is the covariance matrix of the limit distribution. However, we have been unable to find a suitable modification of \cite[Theorem 2.1]{RR09} that enables one to compare the distribution of the random vector $\hat{m}^{(n)}$ with $\mathcal{N}(0, \Lambda)$.

\bibliographystyle{abbrv}
\bibliography{literature}
\end{document}